\documentclass{amsproc}
\usepackage{}

\usepackage[pagebackref, colorlinks = true, linkcolor = blue, urlcolor  = blue, citecolor = red]{hyperref}

\usepackage{amsmath,latexsym,amssymb,amsfonts}
\usepackage{amscd,amssymb,epsfig}
\usepackage{amsthm}
\usepackage{graphicx}
\usepackage[utf8]{inputenc}
\usepackage{tikz}
\usetikzlibrary{positioning}
\usepackage{epstopdf} 
\usepackage{epsfig}
\usepackage{tabularx}

\newtheorem{theorem}{Theorem}[section]
\newtheorem{theoremdefinition}{Theorem/Definition}[section]

\newtheorem{examples}[theorem]{Examples}

\newtheorem{prop}[theorem]{Proposition}

\newtheorem{definition}[theorem]{Definition}
\newtheorem{remark}[theorem]{Remark}

\newtheorem{remarks}[theorem]{Remarks}

\newcommand\BX{\mathbf{X}}

\newcommand\PP{\mathbb{P}}
\newcommand\EE{\mathbb{E}}
\newcommand\FF{\mathbf{F}}
\newcommand\CC{\mathbb{C}}

\newcommand\RR{\mathbb{R}}

\newcommand\cA{\mathcal{A}}

\newcommand\cB{\mathcal{B}}

\newcommand\cD{\mathcal{D}}

\newcommand\cM{\mathcal{M}}
\newcommand\cR{\mathcal{R}}

\newcommand\ee{\varepsilon}

\newcommand\SP{\mathcal{P}}
\newcommand\NCP{\mathcal{NC}}

\numberwithin{equation}{section}

\begin{document}

\title{A General Solution to (Free) Deterministic Equivalents}

\author{Carlos Vargas}
\address{Technische Universitaet Graz \\ Steyrergasse 30 III, 8010, Graz, Austria}
\curraddr{}
\email{carlos@math.uni-sb.de}
\thanks{The author was supported by the Deutsche Forschungsgemeinschaft (DFG) through the project SP419/8-1 and by the Austrian Science Foundation (FWF), project P25510-N26}



\date{\today}

\begin{abstract}
We give an algorithm to compute the asymptotics of the eigenvalue distribution of quite general matricial central limit theorems. The central limits are the so called free deterministic equivalents, which in turn are operators whose Cauchy transforms are the solutions to the equations which define very general deterministic equivalents (a la Girko). Our algorithm is based on the one of Belinschi, Mai and Speicher \cite{BMS13} and the possibility to extend it to more general, operator-valued situations (in particular, to Benaych-Georges rectangular spaces \cite{BG09}).
\end{abstract}
\maketitle

\section{Introduction}

In this work we survey on the different techniques from free probability which are used to study describe the asymptotic spectrum of a quite large class of random matrices, including those very recently used to model wireless communications (see Chapter 6 of \cite{CoDe11} or \cite{LXX15}). The theory of free probability \cite{Vo85} is getting more and more robust and it is quite difficult now to survey on all the aspects that make it such an useful tool for understanding the asymptotic distributions of matrix ensembles. 

For this reason, we develop here only those aspects which lead to a quite direct derivation of fixed point equations for computing asymptotic distributions of Hermitian random matrix models.

The general model is described by a (non-commutative) polynomial $P$ evaluated on deterministic matrices and random (Wigner and Haar) matrices. These models where described in \cite{SpVa12} and were shown to correspond to the solutions of the (recently more recurrent) notion of a deterministic equivalent (DE) for the Cauchy-Stieltjes\footnote{For a probability measure, we will use the Cauchy transform $G_{x}(z):=\tau((z-x)^{-1})$, which just the negative of the Stieltjes-transform $\tau((x-z)^{-1})$} transform, which go back to Girko \cite{Gi01}.

Our method is based on properties of Cauchy transforms of operators which can be very neatly described in terms of moments, which in turn are well behaved if we restrict to normal matrices, and in particular, to self-adjoint matrices. We thus require the polynomial to be self-adjoint after being evaluated by a certain tuple of random and deterministic matrices. The tools for the non-selfadjoint case are being developed in \cite{BSS15}. The matrices that we insert in the polynomial, however, are not required to be self-adjoint. This is one of the main directions in which we extend \cite{BMS13}, which was up-to-now, the only algorithm general enough to deal with large classes of polynomials. Historically, earlier works were usually devoted to study a specific polynomial (and sometimes even with specific inputs).

A second direction is that we do not ask the input matrices to have a fixed size. We only require that the polynomial multiplies these matrices in such a way that all the summands are of the same size. In particular, the condition on the sizes of the matrices to be proportionally large as $N\to\infty$ (as it is ussually assumed when approximating by models by their deterministic equivalent) can be very effectively captured using the formalism of \cite{BG09}.

The equations obtained by our method allow to draw the distributions of most of DE's from \cite{CoDe11}.
For some models, such as the Wigner matrices with variance profile, our algorithm fails to be numerically efficient. It is however, theoretically correct and the FDE corresponds to the operator resulting from substituting independent complex gaussians by free circulars, as described already in \cite{Sh96}. We should point to another recent, quite general application of operator-valued free probability, which describes the asymptotics of block-modified random matrices in terms of Choi Matrices \cite{ANV} and relies on the free multiplicative convolution \cite{BSTV14}.

In some very broad sense, our method allows to compute Gaussian distributions of a Central limit theorem, for which Wigner's semicircle law is a very special, but fundamental case. We will concentrate here merely on describing such central limits, and not on the qualitative aspects of this convergence. 

For more qualitative aspects, such as the analysis of fluctuations, or almost sure convergence (in empirical eigenvalue distributions), a much deeper analysis of moments must be performed. The main aspects of this analysis can already be observed in the proof of Wigner's semicircle law. For this reason, we begin our survey with a sketch of Wigner's semicircle law for the Gaussian unitary ensemble. Then we only point out the main ideas behind the different generalizations, and we refer to the works where such deeper analysis is performed (for the Wigner case in \cite{MiSp12} and for the Haar case in \cite{CoSn06}). 

Once that the contribution of the basic components (Haar and Wigner matrices and deterministic matrices) is understood, we can apply our knowledge to a specific polynomial to obtain its asymptotic distribution. Although we do not perform this here, a more detailed study of the specific interactions (mixed moments) between the constituting matrices, affects the convergence to the limit. Hence, the polynomial $P$ plays a huge role in this analysis, as it ``decides'' the rate in which each specific monomial (or cumulant), will appear while computing the moments of $P$. 

For instance, very generally speaking, the self-adjoint polynomials $Z+Z^*$ and $ZZ^*$ have different behaviors when the arguments are replaced by random and deterministic matrices. The moments of the polynomial $ZZ^*$ are all alternating in $Z$, which immediately leads to better convergence properties. The shape $ZZ^*$ is recurrent in the models of \cite{CoDe11} and this explains the better convergence properties that the authors obtain in the original works. 

In particular, a deeper analysis of mixed moments allows to understand the fact that in some of these models one is allowed to replace the deterministic matrices by random matrices with bounded operator norms, or by matrices with some (weaker) tightness condition.

\subsection{Statement of results}
One of most natural ways of constructing a random matrix is to let each entry be an independent copy of a given random variable $X$. The distribution of $X$ induces a probability measure $\PP$ on subsets $E \subseteq M_N(\CC)$ of matrices. Particularly nice is the matrix $Z=Z_N:=(\frac{1}{\sqrt{N}}z_{ij})_{i,j\leq N}$ with independent standard complex Gaussian entries. Such random matrices are called (non-self-adjoint) Gaussian matrices.

From a Gaussian Matrix $Z_N$ there are two immediate ways to a self-adjoint random matrix: The Wigner matrix $X_N=Z_N+Z_N^*$ and the Wishart Matrix $W_N=Z_NZ_N^*$ (where $A^*$ denotes the Hermitian transpose of the matrix $A$).

Wigner started the study of the asymptotic eigenvalue distribution (AED) of random matrices by establishing the convergence of the AED of $X_N$ to the semicircle law \cite{Wi58}. Later, it was shown that such convergence is universal (a kind of Central Limit Theorem), as it holds even if we replace the Gaussian distribution by any other centered distribution (under some mild moment constrains).\footnote{To avoid an overwhelming terminology, we will simply call ``Wigner Matrices'' to both the self-adjoint and the non-self adjoint matrices with centered i.i.d entries and ``Gaussian matrices'' if the entries are Gaussian. We will distinguish the two cases by using $Z$ for the non-self-adjoint matrices and $X$ for the self-adjoint ones.}

Marcenko and Pastur \cite{MaPa67} studied the second case (and several generalizations of it). They considered first the model $W_N=Z_NZ_N^*$, where $Z_N$ is a $N\times n$ random matrix with independent centered complex Gaussian entries with variance $1/n$. If $N/n\to \lambda\in(0,\infty)$, they showed that the AED of $W_N$ converges to the Marchenko-Pastur law $\nu$ which is given by
$$\nu=\left\{\begin{array}{ll}
(1-\lambda)\delta_0 + \tilde \nu, & \text{if } 0\leq \lambda\leq 1,\\
\tilde \nu, & \text{if } 1< \lambda.   
      \end{array}\right.
$$
where $\delta_z$ denotes the Dirac mass at $z\in\CC$ and $\tilde \nu$ is the measure supported on the interval $[(1-\sqrt{\lambda})^2,(1+\sqrt{\lambda})^2]$, with density
$$\mathrm{d}\tilde \nu(t)= \frac{1}{2\pi t}\sqrt{4\lambda-(t-1-\lambda)^2}dt.$$

Equivalently, the measure $\nu$ is characterized by its Cauchy-Stieltjes transform $G_{\nu}$ which solves the equation 
$$G_{\nu}(z)=(z-\frac{1}{1-\lambda G_{\nu}(z)})^{-1}.$$
The Cauchy-Stieltjes transform of a random variable $X$ with distribution $\mu$ is defined as $$G_{\mu}(z):=G_X(z):=\EE((z-X)^{-1}),$$
wherever the inverse of $(z-X)$ exists. If $X$ is supported on the real line, its Cauchy-transform is defined on the whole complex upper half-plane, and its distribution can be recovered from $G_X$ by performing a Stieltjes inversion: $$\mathrm{d}\mu(t)=\lim_{\varepsilon\downarrow0}-\frac{1}{\pi}\Im G_X(t+i\varepsilon).$$

As the complexity of the model grows, explicit expressions for the densities quickly become intractable and one can only hope to find equations which determine the Cauchy Transforms of the distributions. In \cite{MaPa67}, they considered also the model $Z_NT_nZ_N^*$, where $Z_N$ is as before and $T_n$ is a self-adjoint deterministic matrix, such that the eigenvalue distribution $(\mu_{T_n})$ of $T_n$ converges to a given probability measure $\mu$. For this case, the Cauchy transform of the limiting distribution $\nu(\mu)$ satisfies the equation
\begin{equation}\label{DEMP}
G_{\nu(\mu)}(z)=(z-\int_{\RR}\frac{x\mathrm{d}\mu{(x)}}{1-x\lambda G_{\nu(\mu)}(z)})^{-1}
\end{equation}

Note that the equations for $G_{\nu}$ and $G_{\nu(\mu)}$ are both non-random and the second equation depends on $T_n$ only through $\mu$. In order to eliminate randomness from the equations, it is crucial that we let $N\to \infty$ (and hence also $n\to\infty$).

For finite $N$ it is much harder to derive the exact eigenvalue distributions. However, the most recent models for wireless communications involve deterministic matrices with fixed finite sizes (which depend, for example, on the number of receiving and transmitting antennas). 

The heuristics behind the method of deterministic equivalents (DE) is that, if the matrices involved in the model are large enough, one should still be able to use the asymptotic expressions for the Cauchy transforms to obtain an approximation of the desired distribution.

For example, a DE for the finite dimensional model $Z_NT_nZ_N^*$ is obtained by replacing the limiting deterministic data $(\lambda,\mu)$ by the finite data $(Nn^{-1},\mu_{T_n})$ in equation (\ref{DEMP}): 
$$G_{N}(z)=(z-\int_{\RR}\frac{x\mathrm{d}\mu_{T_n}{(x)}}{1-xNn^{-1}G_{N}(z)})^{-1}.$$

The models $P=\sum_{i=1}^k R_iZ_iT_iZ_i^*R_i^*$ and $Q=\sum_{i=1}^k R_iU_iT_iU_i^*R_i^*$, where the $R$'s and the $T$'s are deterministic and the $Z$'s and $U$'s are, respectively, independent Gaussians and Haar(-distributed unitary) matrices\footnote{We will simply refer to these in the future as ``Haar matrices''}, give further generalizations of the Wishart ensemble. 

The main contribution of this survey will be an algorithm to approximate the distributions of very general polynomials on deterministic, Gaussian and Haar matrices of different sizes. The models $P$ and $Q$, are, nevertheless, illustrative enough and they will serve as main examples throughout this work. 

In \cite{CDS11}, it was shown that a DE for the model $P$ is given as the solution of the system of equations:
\begin{equation}\label{CoDeGau}
m_N(z)=\frac{1}{N}\mathrm{Tr}(zI_N-\sum_{j=1}^k \int_{\RR}\frac{x_j\mathrm{d}\mu_{T_j}{(x_j)}}{1-x_jNn_j^{-1} e_j(z)}R_jR_j^*)^{-1} 
\end{equation}
where $n_j$ is the size of $T_j$ and
$$e_i=\frac{1}{N}\mathrm{Tr}R_iR_i^*(zI_N-\sum_{j=1}^k \int_{\RR}\frac{x_j\mathrm{d}\mu_{T_j}{(x_j)}}{1-x_jNn_j^{-1} e_j(z)}R_jR_j^*)^{-1},$$

A similar system of equations was provided for $Q$ in \cite{CHD11}. The method of deterministic equivalents was shown to work for some other matrix models (see Chapter 6 of \cite{CoDe11} for a survey on these). The equations obtained depended in an ad-hoc way on the specific model in question and the models treated are not so diverse (in terms of the non-commutative polynomial on which the model is based). 

In \cite{SpVa12} we proposed a new approach to deterministic equivalents: Instead of considering approximations of the distributions of the matrix models at the level of Cauchy transforms, we approximate the models themselves at the level of operators, inspired by Voiculescu's free probability theory (\cite{Vo91}, see also \cite{BG09}).  

For example, from the matrix models $P$ and $Q$ one can construct the blown-up models $$P_m=\sum_{i=1}^k R_i^{(m)}Z_i^{(m)}T_i^{(m)}(Z_i^{(m)})^*(R_i^{(m)})^*,\quad Q_m=\sum_{i=1}^k R_i^{(m)}Z_i^{(m)}T_i^{(m)}(Z_i^{(m)})^*(R_i^{(m)})^*,$$ where the $Z_i^{(m)}$'s (resp. $U_i^{(m)}$'s) are again independent Wigner matrices (resp. Haar  matrices) and $A^{(m)}:=A\otimes I_m$ for each deterministic matrix $A$, so that sizes of all the involved matrices are scaled by $m$.

The collection $(R_1^{(m)},T_1^{(m)},Z_1^{(m)}, U_1^{(m)},\dots,R_k^{(m)},T_k^{(m)}, Z_k^{(m)},U_k^{(m)})$ of blown-up matrices converges (in joint non-commutative distribution as $m\to \infty$) to a very specific collection of operators $(R_1,T_1,c_1,u_1\dots,R_k,T_k,c_k,u_k)$ defined in terms of operator-valued free probability. In particular, the AED of $P_m$ and $Q_m$ converge, respectively, to the spectral distribution of the operators $P_{\infty}:=\sum_{i=1}^k R_ic_iT_ic_i^*R_i^*$ and $Q_{\infty}:=\sum_{i=1}^k R_iu_iT_iu_i^*R_i^*$ (see Theorem/Definition \ref{rectfreeness} and Fig. \ref{figCHD}).

\begin{figure}\centering \label{figCHD}
\begin{minipage}{0.45\linewidth}\includegraphics[scale=0.18]{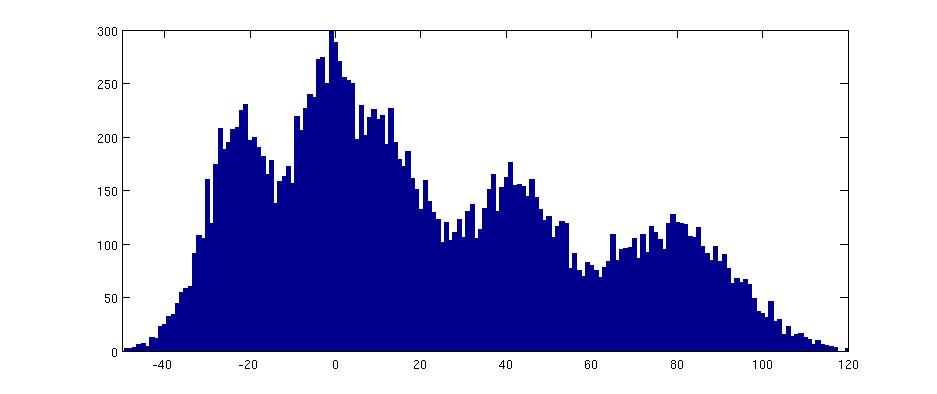}
\end{minipage}
\begin{minipage}{0.45\linewidth}\includegraphics[scale=0.18]{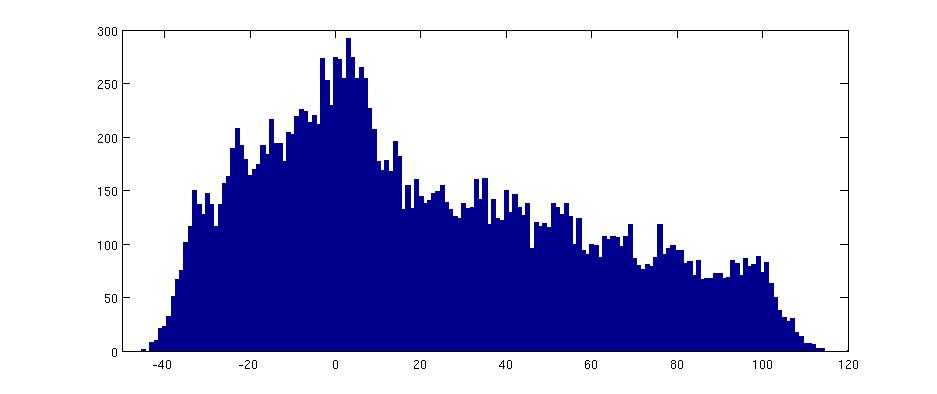} \end{minipage}
\begin{minipage}{0.45\linewidth}\includegraphics[scale=0.18]{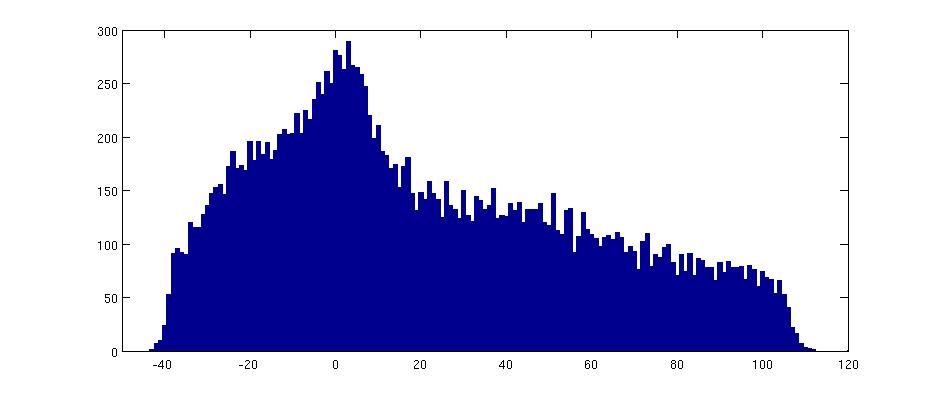} \end{minipage}
\begin{minipage}{0.45\linewidth}\includegraphics[scale=0.18]{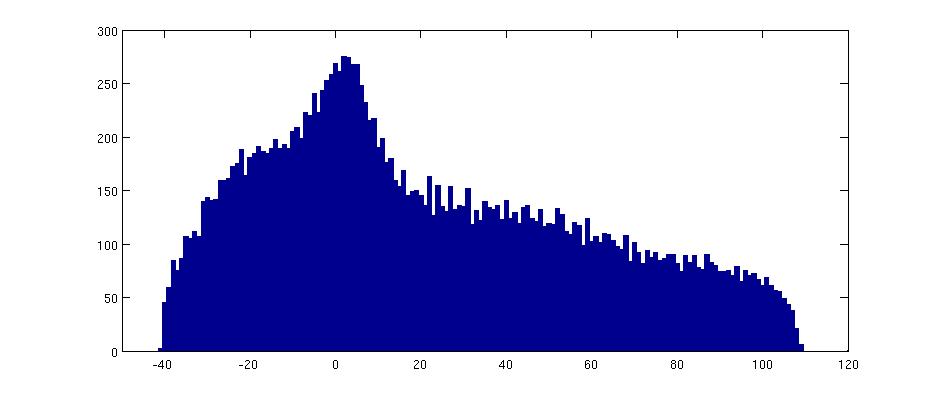}
\end{minipage}
\caption{20000 Eigenvalues of $Q_m$ for $Q=Q_1=R_1U_1T_1U_1^*R_1^*+R_2U_2T_2U_2^*R_2^*$, where $R_1,T_1,R_2,T_2$ are deterministic matrices of sizes $5\times 8$, $5\times 4$, $8\times 8$ and $4\times 4$ and $U_1,U_2$ are (chopped) Haar-Unitaries. Here $m=$ $1$ (up-left), $3$ (up-right), $10$ (down-left), $40$ (down-right).}\label{AEVCou}\end{figure}

In \cite{SpVa12} we then used the combinatorial machinery of free probability (\cite{Sp98}, \cite{NSS02}) to show that the DE (e.g. the solution of the equations (similar to \ref{CoDeGau}) for $Q_N$) is exactly the Cauchy-Transform of the \emph{Free Deterministic Equivalent} $Q_{\infty}$. Moreover, the heuristics behind the method of deterministic equivalents are justified and formalized as a consequence of asymptotic freeness: $Q_m\to Q_{\infty}$ in distribution, but $Q=Q_1$ is already close to the limit if the involved matrices are large.

The same techniques can be applied for $P$ and for the rest of the models in \cite{CoDe11} which are obtained by evaluating a fixed non-commutative polynomial on deterministic, Wigner and Haar matrices. The advantage of our approach is that the definition of the FDE can be easily extended to any arbitrary non-commutative polynomial. New models for wireless communications have been proposed using our formalism \cite{LXX15}. However, up to know, the actual computation of the distribution of $P_{\infty}$ was still performed in an ad-hoc way.

\begin{figure}\label{FigDist}
\centering\includegraphics[scale=0.30]{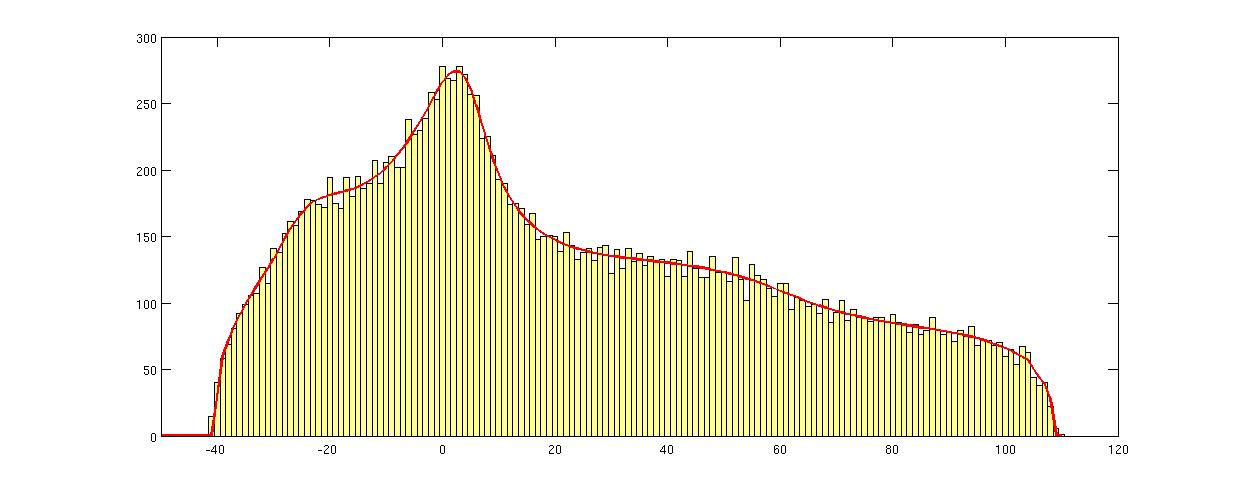} 
\caption{20000 eigenvalues from 100 realizations of $Q_{40}$ (histogram) vs distribution of the free deterministic equivalent (FDE) ${Q_{\infty}}$, computed with our algorithm (solid).}\label{AEDCou2}\end{figure}

In this note, we give a general solution to draw the distribution of a very large class of FDE's (Fig. 2), following the algorithm in \cite{BMS13}, which in turn relies heavily on the analytic subordination phenomena to deal with operator-valued free convolutions (developed in different instances by Biane \cite{Bi98}, Voiculescu \cite{Vo00, Vo02}, Belinschi and Bercovici \cite{BeBe07}), and the linearization trick, which was already suggested by Voiculescu, but became quite prominent after the work of Haagerup and Thorbjornsen (\cite{HaTh05}, see also \cite{An11}). 

Our new algorithm lifts the one in \cite{BMS13} to the operator-valued level and relaxes some unnecessary self-adjointness conditions.

This survey is organized as follows: In Section \ref{SecWigner} we give a sketch of the proof of Wigner's Theorem. Section \ref{SecOVFP} includes all the machinery on Operator-valued free probability which is required to formulate our general algorithm and to establish the correspondence between the known DEs and our FDE's. In Section \ref{SecLinTri} we implement our algorithm to treat several examples (drawn from the literature on wireless communications). 

\subsection*{Acknowledgements}
The author thanks the free probability work group in Saarbrücken for influential discussions (specially to Tobias Mai and Roland Speicher, for discussions concerning the algorithm in \cite{BMS13}).  

\section{Wigner's Theorem}\label{SecWigner}

In this Section we sketch the proof of the simplest version of Wigner's semicircle law, namely, the Gaussian case. This will serve us for future reference when dealing with FDE's (which can be thought as very broad generalizations of this fundamental result).

\subsection{Moments vs cumulants}

Let $X:\Omega\to \RR$ be a random variable in a probability space $(\Omega,\mathcal{F},\PP)$.
The moments of $X$ are the values $$\EE(X^n)=\int_{\Omega}(X(\omega))^n d\PP(\omega).$$
For a large class of random variables, which includes Gaussian random variables and random variables with compact support, the moments of $X$ determine its probability distribution.
For several random variables $X_1,\dots, X_k:\Omega \to \RR$, we may consider the mixed moments
$$\EE(X_1^{n_1}\cdots X_k^{n_k})=\int_{\Omega}(X_1(\omega))^{n_1}\cdots (X_k(\omega))^{n_k} d\PP(\omega).$$

If $X_1,\dots, X_k$ are determined by their moments, then the stochastic independence of $X_1,\dots, X_k$ is equivalent to the fact that, for all $n_1,\dots,n_k\geq 0$, the mixed moments factorize 
\begin{equation}\label{clindep}
\EE(X_1^{n_1}\cdots X_k^{n_k})=\EE(X_1^{n_1})\cdots \EE(X_k^{n_k}). 
\end{equation}
For a collection of random variables $X_1,\dots ,X_k$, we define the (multivariate) classical cumulants $K_n(X_{i_1},\dots X_{i_n}),$ $n\leq 1$, $i_1,\dots ,i_n\leq k$, recursively as the collection of multi-linear functionals $(K_n)_{\geq 1}$ which satisfy the moment-cumulant formula:
$$\EE(X_{i_1}\cdots X_{i_n})=\sum_{\pi\in \SP(n)}K_{\pi}(X_{i_1},\cdots ,X_{i_n}),$$
where:
\begin{definition}\label{partitions}
$\SP(n):=\SP([n])$ are the set partitions; $\SP(n)$ is the power set of $[n]:=\{1,2,\dots,n\}$ and each element $\pi=\{V_1,\dots,V_{|\pi|}\}\in\SP([n])$ decomposes the set $[n]=V_1\cup V_2\cup\dots \cup V_{|\pi|}$ into non-empty, pairwise disjoint subsets (``blocks``) $V_1,\dots,V_{|\pi|}$. We write
\[
K_\pi(X_1,\ldots,X_n):=\prod_{V\in \pi}K_{|V|}(X_V), 
\]
where we use the notation
$$
K_{|V|}(X_V):= K_m(X_{v_1}, \ldots, X_{v_m})
$$  
for each block $V=\{v_1,\ldots,v_m\}\in \pi$, $v_1<\cdots<v_m$. 
\end{definition}
For example, $K_1(X_i)=\EE(X_i)$ is simply the mean and $K_2(X_i,X_j)=\EE(X_iX_j)-\EE(X_i)\EE(X_j)$ is the covariance.

For an ordered tuple $\BX=(X_1,\dots,X_k)$ of random variables, we call $$\Phi_m^{\BX}:=\Phi_m=\{(i_1,\dots,i_m)\mapsto \EE(X_{i_1}\dots X_{i_m}): i_1,\dots,i_m\leq k\}$$ the $m$-th order mixed moments of $(X_1,\dots,X_k)$. Analogously we define $$\Psi_m^{\BX}:=\{(i_1,\dots,i_m)\mapsto K_m(X_{i_1},\dots ,X_{i_m}): i_1,\dots,i_m\leq k\}$$ the $m$-th order cumulants. The collection of moment maps $(\Phi_m^{\BX})_{m\leq n}$ contains exactly the same information as the collection of mixed cumulants $(\Psi_m^{\BX})_{m\leq n}$. However, cumulants seem to organize statistical information in a nicer way. 

A real random variable $X$ is the constant random variable $X=c\in \RR$ iff all cumulants of degree $n\geq 2$ vanish and $\kappa_1(X)=c$. In fact, it is not hard to see that, if we input a constant in any of the arguments of a cumulant of order $k\geq 2$, then the cumulant must vanish, independently from the position of the constant argument and the rest of the arguments.

In terms of cumulants, the simplest (non-constant) random variables are the Gaussian random variables: $X$ has the Normal distribution $\mathcal{N}(\mu,\sigma^2)$ iff all cumulants of degree $n\geq 3$ vanish, $K_2(X,X)=\sigma^2$ and $K_1(X)=\mu$. 

Two random variables $X,Y$ are independent if and only if all the mixed cumulants (i.e. $K_2(X,Y),K_2(Y,X),K_3(X,X,Y),K_3(X,Y,X),K_3(X,Y,Y),\dots$ etc.) vanish. This implies in particular that the cumulants of $X+Y$ are simply $$K_n(X+Y,X+Y,\dots,X+Y)=K_n(X,\dots,X)+K_n(Y,\dots,Y).$$

Hence, cumulants can be used to compute additive convolutions (in fact, the cumulants are related to the coefficients of Fourier transforms).

If $X,Y$ are independent standard Gaussian ($\mathcal {N}(0,1)$) random variables, then $Z:=X+iY$ has the standard complex Gaussian distribution, which can also be characterized in terms of the mixed cumulants of $Z$ and $Z^*:=\bar Z$. The only non-vanishing cumulants are
$$K_2(Z, Z^*)=K_2(Z^*, Z)=1.$$

Wick's formula for independent complex Gaussians $Z_1,\dots,Z_k$, states that, for any $\varepsilon=(\varepsilon_1,\dots, \varepsilon_{n})\in \{1,*\}^n$, we have
\begin{equation}\label{Wick}
\EE(Z_{i_1}^{\varepsilon_1}\dots Z_{i_n}^{\varepsilon_n})=\sum_{\pi\in\SP_2(n)}K_{\pi}(Z_{i_1}^{\varepsilon_1},\dots,Z_{i_n}^{\varepsilon_n}),
\end{equation}
where $\SP_2(n)\subset \SP(n)$ denotes the subset of matchings (i.e. partitions $\pi\in\SP(n)$ such that every $V\in\pi$ has exactly $2$ elements). 
We observe that the Wick formula for this case can be easily derived from the moment-cumulant formula: the restriction to pairings $\SP_2(n)\subset \SP(n)$ follows from the fact that only the second order cumulants of standard complex Gaussians may not vanish. Furthermore, any block $(r,s)$ of $\pi$ should not match independent variables (since the mixed cumulant, and hence $K_\pi$ would vanish).

However, going over cumulants to compute $\EE(Z_{i_1}^{\varepsilon_1}\dots Z_{i_n}^{\varepsilon_n})$ seems more like a detour in this case. We could simply compute $\EE(Z_{i_1}^{\varepsilon_1}\dots Z_{i_n}^{\varepsilon_n})$ by  a direct application of the factorization in eq. (\ref{clindep}). Our Wick formula (\ref{Wick}) will be very useful when we go over to random matrices.

\subsection{Gaussian matrices and Wigner's semicircle law}

Let $Z=Z_N:=(\frac{1}{\sqrt{N}}z_{ij})_{i,j\leq N}$ be a matrix with independent standard complex Gaussian entries $(z_{ij})_{i,j\leq N}$ (the choice of the normalization $\frac{1}{\sqrt{N}}$ will be clear later).

In \cite{Wi58} Wigner described the asymptotic eigenvalue distribution (as $N\to\infty$) of the (necessarily real) eigenvalues of $X_N=(Z_N+Z_N^*)/\sqrt{2}$ (where $A^*$ denotes the Hermitian transpose of $A$).

Since $X(\omega)=(X(\omega))^*$ for any realization of $X$, we can diagonalize $$X(\omega)=U(\omega)D(\omega)(U(\omega))^*,$$ where $D=diag(\lambda_1(\omega),\dots,\lambda_N(\omega))$ and hence for all $k\geq 0$, we have
\begin{eqnarray*}
\frac{1}{N}\mathrm{Tr}(X^k(\omega))=\frac{1}{N}\mathrm{Tr}((U(\omega)D(\omega)(U(\omega))^*)^k)&=&\frac{1}{N}\mathrm{Tr}(U(\omega)D(\omega)^k(U(\omega))^*)\\
&=&\frac{1}{N}\mathrm{Tr}(D(\omega)^k)\\
&=&\frac{1}{N}\sum_{i\leq N}(\lambda_i(\omega))^k.
\end{eqnarray*}
For each $\omega$, expression $\frac{1}{N}\sum_{i\leq N}(\lambda_i(\omega))^k$ can be identified as the $k$-th moment of the real random variable $\Lambda(X(\omega))$, with discrete probability measure $\mu_{X(\omega)}$ which assigns a mass of $1/N$ to each eigenvalue of $X(\omega)$. The averaged eigenvalue distribution (AED) $\mu_{X}$ is the distribution of the random variable $\Lambda(X)$ obtained by averaging all such $\Lambda(X(\omega))$, $\omega\in \Omega$ against $\PP$. More specifically, it is the probability measure $\mu_{X_N}$ with $k$-th moment equal to $\frac{1}{N}\EE(\sum_{i\leq N}\lambda_i^k)$. We want to describe $\mu_{X_N}$ when the size $N\to\infty$. Let us denote $\tau_N:=\frac{1}{N}\mathrm{Tr}$.

Instead of computing the (rather complicated) joint distributions of $(\lambda_1,\lambda_2,\dots, \lambda_N)$, we compute the moments $$\EE\circ \tau_N(X^k(\omega))=\frac{1}{N}\EE(\sum_{i\leq N}\lambda_i^k),$$ which depend on the entries of our matrices in a polynomial way, hoping that we will be able to identify them as moments of a probability measure. 
Indeed, this will be the case, and for this, it will be convenient to work first directly with $Z$ and $Z^*$ and symmetrizing only at the very end. 

In general, we would like to compute, for all $k\geq 1$ and every $\varepsilon=(\varepsilon_1,\dots, \varepsilon_{n})\in \{1,*\}^k$
$$\EE\circ \tau_N(Z^{\varepsilon_1},\dots,Z^{\varepsilon_k}).$$

As an example, let us consider the fourth order mixed moment $\varepsilon=(1,*,*,1)$, we have
\begin{equation}
\EE\circ \tau_N(ZZ^*Z^*Z)=\frac{1}{N^3}\sum_{i_1,\dots,i_4\leq N}\EE(z_{i_1i_2}\bar z_{i_3i_2}\bar z_{i_4i_3}z_{i_4i_1}) 
\end{equation}
By Wick's Formula, 
$$\EE(z_{i_1i_2}\bar z_{i_3i_2}\bar z_{i_4i_3}z_{i_3i_1})=\sum_{\pi\in\SP_2(4)}K_{\pi}(z_{i_1i_2},\bar z_{i_3i_2},\bar z_{i_4i_3},z_{i_4i_1}).$$

There are $3$ pairings of $\{1,2,3,4\}$, namely $\pi_1=\{\{1,2\}\{3,4\}\}$, $\pi_2=\{\{1,3\}\{2,4\}\}$, $\pi_3=\{\{1,4\}\{2,3\}\}$. Since the $z_{ij}$'s are complex Gaussian random variables, $\pi_3$ will vanish (independently of the choice of $i_1,\dots,i_4$) since it will never match a $z_{ij}$ with $\bar z_{ij}$, which is a necessary condition for the cumulant not to vanish. Hence
\begin{eqnarray}
\EE\circ \tau_N(ZZ^*Z^*Z)&=&\frac{1}{N^3}\sum_{i_1,\dots,i_4\leq N}K_{\pi_1}(z_{i_1i_2},\bar z_{i_3i_2},\bar z_{i_4i_3},z_{i_4i_1})\\
&+&\frac{1}{N^3}\sum_{i_1,\dots,i_4\leq N}K_{\pi_2}(z_{i_1i_2},\bar z_{i_3i_2},\bar z_{i_4i_3},z_{i_4i_1})
\end{eqnarray}
Since the entries are independent standard complex Gaussians, each partition imposes restrictions on the indices for the cumulants not to vanish, namely
\begin{eqnarray*}
K_{\pi_1}(z_{i_1i_2},\bar z_{i_3i_2},\bar z_{i_4i_3},z_{i_4i_1})=K_{2}(z_{i_1i_2},\bar z_{i_3i_2})K_{2}(\bar z_{i_4i_3},z_{i_4i_1})&=&\delta_{i_1i_3}\delta_{i_3i_1},\\
K_{\pi_2}(z_{i_1i_2},\bar z_{i_3i_2},\bar z_{i_4i_3},z_{i_4i_1})=K_{2}(z_{i_1i_2},\bar z_{i_4i_3})K_{2}(\bar z_{i_3i_2},z_{i_4i_1})&=&\delta_{i_1i_4}\delta_{i_3i_2}\delta_{i_3i_4}\delta_{i_2i_1}, 
\end{eqnarray*}
Hence we need only to count the number of free indices in order to obtain the contribution of each partition. For this case we obtain
$$\EE\circ \tau_N(ZZ^*Z^*Z)=1+1/N^2.$$

As $N\to\infty$ only the contribution of $\pi_1$ will survive. For a general moment of order $k$, an easy inductive argument shows that a pairing $\pi\in \SP_2(k)$ can only contribute in the limit if $\pi \in \NCP_2(k)\subset \SP_2(k)$ is a non-crossing pairing (i.e. there is no quadruple $1\leq a<b<c<d\leq k$ such that $a,c\in V_i,b,d\in V_j$ where $V_i\neq V_j$ are blocks of $\pi$). In addition, we must have that, for each matching $\{a,b\}\in\pi$, $\varepsilon_a\neq \varepsilon_b$ (in contrast to $\pi_3$ above). 
Hence, for computing the asymptotics of a general moment we need to find
\begin{equation}
\frac{1}{N^{1+k/2}}\sum_{\substack{i_1,\dots,i_k\leq N \\ \pi\in\NCP_2(k)}}K_{\pi}(z_{i_1i_2}^{\varepsilon_1},z_{i_2i_3}^{\varepsilon_2},\dots,z_{i_ki_1}^{\varepsilon_k})=|\NCP_{\ee}(k)|,
\end{equation}
where $\pi\in\NCP_{\ee}(k)\subseteq \NCP_2(k)$ iff $\ee_a\neq\ee_b$ for each pair $\{a,b\}\in\pi$.  

Now, since $X=(Z+Z^*)/\sqrt 2$, we have that 
\begin{eqnarray*}\label{eqcirctosemicirc}
\lim_{N\to\infty}\EE\circ \tau_N(X^k)&=&\frac{1}{N^{1+k/2}2^{k/2}}\sum_{\substack{i_1,\dots,i_k\leq N \\ \varepsilon=(\varepsilon_1,\dots,\varepsilon_k)\in\{1,*\}^k \\ \pi\in\NCP_2(k)}}K_{\pi}(z_{i_1i_2}^{\varepsilon_1},z_{i_2i_3}^{\varepsilon_2},\dots,z_{i_ki_1}^{\varepsilon_k}).
\end{eqnarray*}
If we fix a non-crossing pairing $\pi\in \NCP_2(k)$, there are $2^{k/2}$ non-vanishing choices for $\ee$ (for each block of $(a,b)\in\pi$, we can have $(\varepsilon_a,\varepsilon_b)$ equal to either $(1,*)$ or $(*,1)$). Hence, after summing over all free indices and all $\ee\in\{1,*\}^k$, each non-crossing pairing contributes with $1$ to the moment in the limit. It is well-known that the number of non-crossing pairings $|\NCP_2(k)|$ are counted by the Catalan numbers $C_n:=\frac{1}{n+1}\binom{n}{2n}$, which in turn, are the even moments of the semicircular distribution supported on $[-2,2]$ with density $$\mathrm{d}\mu(t)=\frac{1}{2\pi}\sqrt{4-t^2}.$$ Hence the assertion follows.

In the next section we introduce Voiculescu's free probability theory for non-commutative random variables, which allows to treat random matrices as random variables. One of the key ideas is to think of $\tau_N=\frac{1}{N}\EE\circ \mathrm{Tr}$ as a generalization of the expectation $\EE$ in classical probability.

\section{Operator-Valued Free Probability}\label{SecOVFP}

Voiculescu's free probability \cite{Vo85} is a prominent branch of non-commutative probability, where the classical probabilistic notion of moments (with respect to the expectation) is extended to more general linear functionals $\tau:\cA\to \CC$ on a $C^*$-(or $W^*$-)algebra $\cA$. A crucial step in the development of free probability was to further extend this to cover the classical notion of conditional expectation \cite{Vo95}.

\begin{definition}\label{DefOVPS}
Let $\mathcal{A}$ be a unital $*$-algebra and let $\CC\subseteq\mathcal{B}\subseteq\cA$ be a $*$-sub-algebra. 
A $\mathcal{B}${-probability space} is a pair $\left(\mathcal{A},\mathbf{F}\right)$, where  $\mathbf{F}:\mathcal{A}\to\mathcal{B}$ is a {conditional expectation}, that is, a linear map satisfying:
\begin{eqnarray*}
\mathbf{F}\left(bab'\right) &=& b\mathbf{F}(a)b', \qquad \forall b,b'\in\mathcal{B},a\in\mathcal{A}\\
\mathbf{F}\left(1\right) &=& 1. 
\end{eqnarray*} 
\end{definition}

The elements $a\in\cA$ are called (non-commutative) random variables and the pair $(\cA,\FF)$ is called a $\cB$-valued probability space. The situation $\cB=\CC$ is usually highlighted by writing $\tau$ instead of $\FF$ and the pair $(\cA,\tau)$ is then called a (scalar-valued) non commutative probability space (NCPS). 

The main virtue of Definition \ref{DefOVPS} is that we can treat simultaneously classical random variables and deterministic matrices in the same framework. These spaces are important building blocks for many interesting examples.

\begin{examples}
(1). Complex-valued random variables in a classical probability space $(\Omega,\mathcal{F},\PP)$ can be thought as a NCPS $(\cA,\tau)$ with involution given by complex conjugating the random variables and $\tau:=\EE$ being the usual expectation. Any  sub-sigma algebra $\mathcal{H}\subset\mathcal{F}$ induces a conditional expectation. This means a map $X\mapsto \EE(X|\mathcal{H})$, where for any $E\in\mathcal{H}$, $$\int_E \EE(X|\mathcal{H})\mathrm d \PP=\int_E X\mathrm d \PP.$$ In particular for  the trivial sub-sigma algebra $\mathcal{H}_1=\{\emptyset, \Omega\}$ we recover the usual expectation and for $\mathcal{H}_2=\mathcal F$ we obtain the identity map $\EE(X|\mathcal{H}_2)=X$.

(2) The NCPS $(\cA:=\cM_N(\CC), \tau_N:=\frac{1}{N}\mathrm{Tr})$ of complex $N\times N$ matrices with the normalized trace is a $*$-probability space (with involution given by the Hermitian transpose). 
As we saw in the previous section, if $X\in\cM_N(\CC)$ is Hermitian, we may diagonalize $X=UDU^*$ and realize that $$\tau(X^k)=\tau(UD^kU^*)=\tau(D^k)=\frac{1}{N}\sum_{i=1}^N\lambda_i(X)^k=\int_{\RR}x^k\mu_X(dt),$$ where $\lambda_1(X)\leq\dots\leq\lambda_N(X)$ are the eigenvalues of $X$ and $\mu_X=\frac{1}{N}\sum_{i=1}\delta_{\lambda_i(X)}.$ 

Hence $(\tau(a^k))_{k\geq 1}$ are the moments (in the usual, probabilistic sense) of the averaged eigenvalue distribution $\mu_X$. Such a correspondence between moments and probability distributions (on the complex plane) works not only for Hermitian matrices but for general normal matrices (which in particular include all Hermitian, Skew-Hermitian and Unitary matrices).
\end{examples}

Combining (via the algebraic tensor product) a classical probability space $(\cA,\EE)$ with the NCPS of deterministic matrices $(\cM_N(\CC),\tau_N)$ yields the NCPS $(\cA\otimes \cM_N(\CC),\EE\otimes \tau_N)$, which can be identified with the space of matrices with random entries drawn from $\cA$. In this way, random matrices can be treated as non commutative random variables.

\begin{examples}
Let $(\cA,\tau)$ be a NCPS.

(1). Let $p_1,\dots,p_k\in\cA$ be pairwise orthogonal projections with $1_{\cA}=p_1+\dots+p_k$. There exist a unique conditional expectation $\FF:\cA\to \langle p_1,\dots,p_k\rangle$ compatible with $\tau$ in the sense that $\tau\circ \FF=\tau$, explicitly given by $$\FF(a)=\sum_{i\leq k}\tau(p_i)^{-1}\tau(p_iap_i).$$ 

(2). Consider the algebra $M_N(\cA)\cong M_N(\CC)\otimes\cA$ of $N\times N$ matrices with entries in $\cA$.
The maps 
\begin{eqnarray}
\FF_3:(a_{ij})_{ij}&\mapsto&(\tau(a_{ij}))_{ij}\in M_N(\CC),\\
\FF_2:(a_{ij})_{ij}&\mapsto&(\delta_{ij}\tau(a_{ij}))_{ij}\in D_N(\CC),\\
\FF_1:(a_{ij})_{ij}&\mapsto&\sum_{i=1}^N\frac{1}{N}\tau(a_{ii})I_N\in \CC\cdot I_N
\end{eqnarray}
are respectively, conditional expectations onto the algebras $M_n(\CC)\supset D_n(\CC)\supset \CC\cdot I_N$ of constant matrices, diagonal matrices and multiples of the identity. Note that $(M_N(\CC)\otimes\cA,\FF_1)$ is a scalar-valued NCPS.
\end{examples}

The generality of non-commutative random variables allowed Voiculescu to define a parallel notion to (conditional) independence, which he called \emph{freeness} (\emph{with amalgamation}).

\begin{definition}
Let $(\cA,\FF)$ be a $\cB$-probability space and let $\bar{a}:=a-\FF(a)1_{\cA}$ for any $a\in\cA$. 
The $*$-sub-algebras $\cB\subseteq A_1,\dots ,A_k\subseteq \cA$ are $\cB${-free} (or {free over} $\cB$, or {free with amalgamation over} $\cB$) ({with respect to} $\FF$) iff 
\begin{equation}\label{opfreeness}
\FF(\bar{a_1}\bar{a_2}\cdots \bar{a_m})=0, 
\end{equation}
for all $m\geq 1$ and all tuples $a_1,\dots,a_m\in \cA$ such that $a_i\in A_{j(i)}$ with $j(1)\neq j(2)\neq \dots \neq j(m)$ (note that, for example, $j(1)=j(3)$ is allowed).

Subsets $S_1,\dots ,S_k\subset \cA$ are $\cB${-free} if so are the $*$-sub-algebras $\langle S_1,\cB\rangle,\dots,\langle S_k,\cB\rangle$.
\end{definition}

Free probability is a realm, parallel to classical probability, where the factorization of expectations (Eq. \ref{clindep}) given by classical independence is replaced by freeness. Many fundamental theorems from classical probability, such as the convergence to the Central Limit or the Law of small numbers can be translated to the free setting. 

In particular, by just replacing independence by free independence and working algebraically, the free Central Limit (i.e. the limiting distribution of $S_N=\frac{1}{\sqrt N}(\bar a_1+\dots+\bar a_N)$ for free self-adjoint, identically distributed elements $(a_i)_{i\geq 1}$) turns out to be Wigner's semicircle law, whereas the free analog of the law of small numbers is the Marchenko-Pastur distribution (also known as free Poisson), which is the (also universal) limit of singular-value distributions of Wishart matrices.

In his seminal work \cite{Vo91}, Voiculescu constructed free operators inspired by these pioneering results on the asymptotics of random matrices. In order to state his results on asymptotic freeness, we need to define special non-commutative random variables, and the notion of convergence in non-commutative distribution.

\begin{definition} \label{DefCircHaarU}
Let $(\cA,\tau)$ be a NCPS.

(1). A random variable $c\in\cA$ is a circular element iff, for any $k\geq 1$ and $\varepsilon=(\varepsilon_1,\dots,\varepsilon_k)\in \{1,*\}^k$, the mixed moment on $(c,c^*)$ are
$$\tau(c^{\varepsilon_1}c^{\varepsilon_2}\dots c^{\varepsilon_k})=|\NCP_{\ee}(k)|.$$

(2). A random variable $u\in\cA$ is a Haar-unitary iff $u^*u=1=uu^*$ and $\tau(u^k)=0$ for all $k\geq1$.
\end{definition}
Note that $c$ is not a normal operator, (for example, $\tau(ccc^*c^*)=1\neq 2 = \tau(cc^*cc^*)$ which would be equal if $c^*c=cc^*$). However, by Equation \ref{eqcirctosemicirc} (and the discussion afterwards), it is easy to see that the self-adjoint random variable $s=2^{-1/2}(c+c^*)$ has the standard semicircular distribution.
\begin{definition} \label{convdist}
(1). For an ordered tuple $a=(a_1,\dots,a_k)$ of random variables, we define the distribution of $a$ as the collection $\Phi(a):=(\Phi_m)_{m\geq 1}$ of maps $$\Phi_m:\{(i_1,\dots,i_m)\mapsto \tau(a_{i_1}\dots a_{i_m}): i_1,\dots,i_m\leq k\},$$
If $a=(a_1,\dots,a_k)\in \cA _1^k$ and $b=(b_1,\dots,b_k)\in \cA_2^k$ are tuples of random variables in (possibly different) NCPS $(\cA_1,\tau_1)$, $(\cA_2,\tau_2)$, such that $\tau_1(a_{i_1}\dots a_{i_m}) = \tau_2(b_{i_1}\dots b_{i_m})$ for all $m\geq 1$ and all $1\leq i_1,\dots, i_m\leq k$, we say that $a$ and $b$ have the same distribution and we write $a\sim b$

(2). Let $(\cA_N,\tau_N)$, $N\geq 1$, and $(\cA,\tau)$ be $\CC$-probability spaces and let $(a_1^{(N)},\dots,a_k^{(N)})\in \cA_N^k$, $(a_1,\dots,a_k)\in \cA^k$ be such that $$\lim_{N\to\infty}\tau_N((a_{i_1}^{(N)})\cdots (a_{i_r}^{(N)}))=\tau(a_{i_1}\cdots a_{i_r}),$$ for all $r\geq 1$, $1\leq i_1,\dots, i_m\leq k$ Then we say that $(a_1^{(N)},\dots,a_k^{(N)})$ converges in distribution to $(a_1,\dots,a_k)$ and we write $(a_1^{(N)},\dots,a_k^{(N)})\to(a_1,\dots,a_k)$.
\end{definition}

\subsection{Asymptotic freeness of Random Matrices}

\begin{theorem}\label{afreeness}
For each $N\geq1$, let $Z_1^{(N)},\dots,Z_p^{(N)}$ and $U_1^{(N)},\dots,U_q^{(N)}$ be $N\times N$ independent Wigner and Haar unitary matrices. 
Let $D_1^{(N)},\dots,D_r^{(N)}$ be deterministic matrices, such that, for any $k\geq 1$ and $1\leq j_1,\dots,j_k\leq r$ there exist a constant $c(j_1,\dots,j_k)\in \CC$ such that
\begin{equation}\label{sdetcond}
\lim_{N\to\infty}\frac{1}{N}\mathrm{Tr}((D_{j_1}^{(N)})(D_{j_2}^{(N)})\cdots (D_{j_k}^{(N)}))=c(j_1,\dots,j_k). 
\end{equation}
Then, as $N\to \infty$, $$(X_1^{(N)},\dots,X_p^{(N)},U_1^{(N)},\dots,U_q^{(N)},D_1^{(N)},\dots,D_r^{(N)})\to(s_1,\dots,s_p,u_1,\dots u_q,d_1,\dots d_r)$$ where $s_1,\dots,s_p,u_1,\dots u_q,d_1,\dots d_r$ are elements in some $*$-probability space $(\cA,\tau)$ whose joint-distribution is determined by the following conditions:
\begin{itemize}
 \item $c_i$ is a circular element for all $i\leq p$.
 \item $u_i$ is a Haar-unitary for all $i\leq q$.
 \item $\tau(d_{i_1}d_{i_2}\cdots d_{i_k})=c(i_1,\dots,i_k)$, for any $k\geq 1$, $1\leq i_1,\dots,i_k\leq r$.
\item The algebras $\langle s_1\rangle,\dots,\langle s_p\rangle,\langle u_1,u_1^*\rangle,\dots, \langle u_q,u_q^*\rangle,\langle d_1,\dots, d_r \rangle $ are ($\CC$-)free.
\end{itemize}
\end{theorem}

Theorem \ref{afreeness} generalizes Wigner's semicircle law in several directions. It allows us to compute the asymptotic mixed moments of $$(X_1^{(N)},\dots,X_p^{(N)},U_1^{(N)},\dots,U_q^{(N)},D_1^{(N)},\dots,D_r^{(N)})$$ by means of the rule of free independence (\ref{opfreeness}), in terms of the individual asymptotic moments of $X_1^{(N)},\dots, X_p^{(N)}, U_1^{(N)},\dots,U_q^{(N)}$ and the (given) asymptotic mixed moments of $(D_1^{(N)},\dots,D_r^{(N)})$. 
A nice way to understand how these mixed moments are calculated is in terms of free cumulants (see Section \ref{combi}).
 
Our combinatorial proof of the Gaussian case in Section \ref{SecWigner} will be our main reference to indicate how the different generalizations work. 

\begin{itemize}
 \item Relaxing Gaussian condition.
\end{itemize}

Let us consider first a single matrix $Z$. Again, we need to study 
\begin{equation}
\frac{1}{N}\EE\circ \mathrm{Tr}(Z^{\varepsilon_1},\dots,Z^{\varepsilon_k})=\frac{1}{N^{1+k/2}}\sum_{\substack{i_1,\dots,i_k\leq N \\ \pi\in\SP(k)}}K_{\pi}(z_{i_1i_2}^{\varepsilon_1},z_{i_2i_3}^{\varepsilon_2},\dots,z_{i_ki_1}^{\varepsilon_k}).
\end{equation}
For the Gaussian case all cumulants of order different than $2$ vanished and hence our sum ran over $\SP_2(k)$. Then we noticed that only $\pi\in \NCP_2(k)$ actually matter in the limit. 

If $Z$ is no longer Gaussian, we need to consider partitions $\pi\in \SP(k)$ such that all blocks of $\pi$ are of size greater or equal to $2$, but $\pi\in \SP(k)$ needs not in principle to be in $\SP_2(k)$. Blocks of size $1$ are still not allowed because the entries of $Z$ are all centered. The important observation is that the number of different cumulants to be considered depends only on the size $k$ of the moment in question (and does not grow with $N$, which, for each $\pi$, only affects the number of choices for each free index). 

If we allow $\pi$ to have bigger blocks, it is intuitive that the contribution of $\pi$ will vanish in the limit since it will imply that more indices need to be identified. This can be shown by induction.

\begin{itemize}
 \item Several Wigner Matrices.
\end{itemize}
Going from one Wigner matrix to several is not hard. We now need to consider the more general expression, for $j_1,\dots,j_k\leq p$
$$\frac{1}{N}\EE\circ \mathrm{Tr}(Z_{j_1}^{\varepsilon_1},\dots,Z_{j_k}^{\varepsilon_k})=\frac{1}{N^{1+k/2}}\sum_{\substack{i_1,\dots,i_k\leq N \\ \pi\in\SP(k)}}K_{\pi}(z_{i_1i_2}^{\varepsilon_1;j_1},z_{i_2i_3}^{\varepsilon_2;j_2},\dots,z_{i_ki_1}^{\varepsilon_k;j_k}),$$
where now the blocks of $\pi\in \SP(n)$ should also respect the labels indicated by $j=(j_1,\dots,j_k)$. Such restriction can be then carried out to the very end. In the limit, the sum will again run over non-crossing pairings $\pi\in \NCP_{\ee}(n)\subset\NCP_2(n)$ with the additional condition that for all $(a,b)\in \pi$, $j_a=j_b$. We will see later in Section \ref{combi} that this characterizes a free family of circular operators.
\begin{itemize}
 \item Wigner matrices and deterministic matrices.
\end{itemize}
 
If we now allow deterministic matrices to operate between our Wigner matrices we need to compute expressions of the form
$$\frac{1}{N}\EE\circ \mathrm{Tr}(D_{j_0}Z_{j_1}^{\varepsilon_1}D_{j_2},\dots,Z_{j_{2k-1}}^{\varepsilon_k}D_{j_{2k}}).$$
More generally, we are interested in estimating
$$\frac{1}{N^{1+k/2}}\sum_{\substack{i_0,\dots,i_{2k+1}\leq N \\ \pi\in\SP(2k+1)}}K_{\pi}(d_{i_0i_1}^{(j_0)},z_{i_1i_2}^{\varepsilon_1},d_{i_2i_3}^{(j_2)},z_{i_3i_4}^{\varepsilon_2},\dots,z_{i_{2k-1}i_{2k}}^{\varepsilon_k},d_{i_{2k}i_{2k+1}}^{(j_{2k})}).$$
Since the $d_{ij}$'s are constants, $\{\{1\},\{3\},\dots,\{2k+1\}\}\subset \pi$, otherwise the cumulant vanishes. Therefore, we need only to consider partitions of $\SP(2,4,\dots, 2k)\cong \SP(k)$.

As an example, let us assume that we only have a single Gaussian matrix and consider the pairing $\{\{1,3\}\{2,7\}\{4,5\}\{6,8\}\}$ and $\ee=(1,*,*,1,*,1,1,*)$. Again, each block corresponds to a cumulant which imposes identifications of some indices: 

$h_0:=i_0,\quad h_1:=i_1=i_6,\quad h_2:=i_2=i_5,\quad h_3:=i_4=i_{13},\quad h_4:=i_3=i_{14},$ 

$h_5:=i_7=i_{10},\quad h_6:=i_{8}=i_{9},\quad h_7:=i_{11}=i_{16},\quad h_8:=i_{12}=i_{15},\quad h_9:=i_{17}$

The contribution of $\pi$ will be $K_{\pi}(z,\bar z, \bar z,z,\bar z,z,z,\bar z)$ (which is one in this case, and does not depend on $N$ in general), times $\mathrm{Tr}(D_{\pi,\ee})$, where $$D_{\pi,\ee}=\sum_{h_0,\dots,h_9\leq N} d^{(j_0)}_{i_0i_1}d^{(j_2)}_{i_2i_3}\dots d^{(j_{2k})}_{i_{2k}i_{2k+1}}.$$ For our example the sum $D_{\pi,\ee}$ is given by:
\begin{eqnarray}
D_{\pi,\ee}&=&\sum_{h_0,\dots,h_9\leq N}
d^{(j_0)}_{i_{0}i_{1}} d^{(j_{2})}_{i_{2}i_{3}} d^{(j_{4})}_{i_{4}i_{5}} d^{(j_{6})}_{i_{6}i_{7}} d^{(j_{8})}_{i_{8}i_{9}} d^{(j_{10})}_{i_{10}i_{11}} d^{(j_{12})}_{i_{12}i_{13}} d^{(j_{14})}_{i_{14}i_{15}} d^{(j_{16})}_{i_{16}i_{17}}\\
&=&\sum_{h_0,\dots,h_9\leq N}
d^{(j_0)}_{h_{0}h_{1}} d^{(j_{2})}_{h_{2}h_{4}} d^{(j_{4})}_{h_{3}h_{2}} d^{(j_{6})}_{h_{1}h_{5}} d^{(j_{8})}_{h_{6}h_{6}} d^{(j_{10})}_{h_{5}h_{7}} d^{(j_{12})}_{h_{8}h_{3}} d^{(j_{14})}_{h_{4}h_{8}} d^{(j_{16})}_{h_{7}h_{9}}\\
&=&\sum_{h_0,\dots,h_9\leq N} 
(d^{(j_0)}_{h_{0}h_{1}} d^{(j_{6})}_{h_{1}h_{5}} d^{(j_{10})}_{h_{5}h_{7}} d^{(j_{16})}_{h_{7}h_{9}}) (d^{(j_{4})}_{h_{3}h_{2}} d^{(j_{2})}_{h_{2}h_{4}} d^{(j_{14})}_{h_{4}h_{8}} d^{(j_{12})}_{h_{8}h_{3}}) (d^{(j_{8})}_{h_{6}h_{6}} )   \\
&=&D_{j_0}D_{j_6}D_{j_{10}}D_{j_{16}}\mathrm{Tr}(D_{j_{4}}D_{j_{2}}D_{j_{14}}D_{j_{12}})\mathrm{Tr}(D_{j_{8}}),
\end{eqnarray}
Hence 
\begin{eqnarray*}
\frac{1}{N^{1+k/2}}\mathrm{Tr}(D_{\pi,\ee})&=&\frac{1}{N^{5}}\mathrm{Tr}(D_{j_0}D_{j_6}D_{j_{10}}D_{j_{16}})\mathrm{Tr}(D_{j_{4}}D_{j_{2}}D_{j_{14}}D_{j_{12}})\mathrm{Tr}(D_{j_{8}})\\
&=&\frac{1}{N^{2}}[\tau_N(D_{j_0}D_{j_6}D_{j_{10}}D_{j_{16}})\tau_N(D_{j_{4}}D_{j_{2}}D_{j_{14}}D_{j_{12}})\tau_N(D_{j_{8}})]\\
&\to&\frac{1}{N^{2}} c(j_0,j_6,j_{10},j_{16})c(j_4,j_2,j_{14},j_{12})c(j_{8})\\
&\to& 0.
\end{eqnarray*}
The Gaussian case here is notably easier than the general Wigner case. Pair partitions are quite convenient because they identify pairs of indices, which allows us to order the $d_{ij}$ in cycles as we did above. Then one can show inductively that only non-crossing pairings matter in the limit, as these produce the maximum number of cycles.

For the general, non-Gaussian case, where cumulants of order greater than $2$ are allowed, more than two indices may be identified and such cyclic reordering of the $d_{ij}$ is not possible in general. In \cite{MiSp12}, the authors associate a graph to each partition and calculate sharp estimates to prove not only that non-crossing pair partitions are the only non-vanishing contribution, but also stronger forms convergence which depend on more delicate analysis of the contribution of these partitions. Alternatively, one may use concentration of measure \cite{AGZ10}.

\begin{itemize}
 \item Haar Matrices
\end{itemize}

The joint distribution of the entries of a Haar-distributed unitary matrix $U_N=(u_{ij})_{i,j\leq N}$ is quite complicated. The entries $u_{ij}$ are known to fulfill the Wick formula $$\EE(u_{i_1j_1}\dots u_{i_qj_q}\bar u_{i_1'j_1'}\dots \bar u_{i_q'j_q'})=\sum_{\rho,\sigma\in S_q}\delta_{i_1i_{\rho(1)}'}\dots \delta_{i_qi_{\rho(q)}'}\delta_{j_1j_{\sigma(1)}'}\dots \delta_{j_qj_{\sigma(q)}'}Wg(N,\rho \sigma^{-1}),$$ where, for each $N$, the Weingarten function $Wg_N:S_q\to \CC$ is some  class function whose values depend on characters of representations of symmetric groups.

The leading term of the Weingarten function can be obtained from the asymptotic expansion
$$Wg(N,\sigma)=N^{-n-|\sigma|}\prod_{i}(-1)^{l(\alpha_i)-1}C_{(l(\alpha_i)-1)}+O(N^{-n-|\sigma|-2}),$$
where the permutation $\sigma=\alpha_1\dots\alpha_k\in S_n$ is a product of cycles $\alpha_1,\dots,\alpha_k$ of lengths $l(\alpha_i)$, $C_k$ is the $k$-th Catalan number and $|\sigma|$ is the minimum number of transpositions required to express $\sigma$.

One may perform an asymptotic analysis, similar to the ones that we did for the Wigner case, to conclude hat the asymptotic mixed moments of the matrices are computed according to the rules described in Theorem \ref{afreeness} (see \cite{Co03}, \cite{CoSn06}).

\begin{remarks}\label{remafree}

(1). Voiculescu's theorem covers Wishart matrices as well: If $p,c$ are free, $c$ is a circular element and $p$ is a projection, a Wishart matrix of parameter $\lambda> 0$ can essentially be viewed asymptotically as a scalar multiple of $cpc^*$ or $pcc^*p$ (depending on whether $\lambda\leq1$ or not).

(2). There exist a stronger version of freeness, called {second-order freeness}, which allows to control fluctuations of random matrices. Second order freeness is achieved (see \cite{AGZ10} or \cite{MiSp12}) if we slightly strengthen the assumptions on our matrices, by either asking all Wigner random matrices to be Gaussian matrices, or by asking all deterministic matrices to be diagonal. Under such conditions, the empirical eigenvalue distribution of any polynomial $$P(X_1^{(N)},\dots,X_p^{(N)},U_1^{(N)},\dots,U_q^{(N)},D_1^{(N)},\dots,D_r^{(N)})$$ converges almost surely to the spectral distribution of $P(s_1,\dots,s_p,u_1,\dots u_q,d_1,\dots d_r)$.
\end{remarks}

Freeness leads quite immediately to matrix-valued freeness: If $A_1,\dots ,A_k$ are free in $(\cA,\tau)$, then the algebras $M_n(\CC)\otimes A_1,\dots,M_n(\CC)\otimes A_k$ of matrices with entries in $A_1,\dots,A_k$ are in general not free over $\CC$ (w.r.t. $\frac{1}{n}\mathrm{Tr}\otimes \tau$). 
They are, however $M_n(\CC)$-free (w.r.t. $id_{M_n(\CC)}\otimes \tau$). 
Below is a slightly more general assertion of this simple but fundamental result.

\begin{prop}\label{matrixfreeness}
Let $(\cA,\FF)$ be a $\cB$-probability space, and consider the $M_n(\cB)$-valued probability space $(M_n(\CC)\otimes\cA,id\otimes\FF)$. 
If $A_1,\dots,A_k\subseteq\cA$ are $\cB$-free, then $(M_n(\CC)\otimes A_1),\dots,(M_n(\CC)\otimes A_k)\subseteq(M_n(\CC)\otimes\cA)$ are $(M_n(\cB))$-free.
\end{prop}
\begin{proof}
Let $a^{(1)},\dots,a^{(m)}\in M_n(\CC)\otimes\cA$ be such that $a^{(i)}\in M_n(\CC)\otimes A_{j(i)}$ with $j(1)\neq j(2)\neq \dots \neq j(m)$. 
Observe that $$\overline{a^{(i)}}=a^{(i)}-(id\otimes \FF)(a^{(i)})=((a^{(i)}_{rs})-\FF(a^{(i)}_{rs}))_{rs\leq n}=(\overline{a^{(i)}_{rs}})_{rs\leq n}.$$ 
Hence
\begin{equation}
(id\otimes\FF)((\overline{a^{(1)}})\cdots(\overline{a^{(m)}}))=\sum_{i_0,\dots,i_m=1}^n(\FF((\overline{a^{(1)}_{i_0i_1}})(\overline{a^{(2)}_{i_1i_2}})\cdots(\overline{a^{(m)}_{i_{m-1}i_m}})))_{i_0i_m}=0.
\end{equation}
\end{proof}

Now that the main aspects of Voiculescu's asymptotic freeness results have been understood, we use Benaych-Georges generalization to rectangular spaces. For our convenience, the version we present is specialized to the situation where the deterministic matrices are blown-up so that condition \ref{sdetcond} holds trivially. This allows us to directly define free deterministic equivalents.

\begin{theoremdefinition}[\cite{BG09b}/\cite{SpVa12}]\label{rectfreeness}
Let $N,k\geq 1$ be fixed and let $P_1,\dots,P_k$ be pairwise orthogonal projections such that $I_N=P_1+\dots+P_k$.

Let $D_1,\dots,D_r$ be deterministic matrices, with $D_i=P_{h_1(i)}D_iP_{h_2(i)}$ for some $1\leq h_1(i),h_2(i)\leq k$.

For each $m\geq 1$ and $A\in\{P_1,\dots,P_k, D_1,\dots,D_r\}$, let $A^{(m)}:=A\otimes I_m$.
Consider now a space $\cA_{Nm}$ of $Nm\times Nm$ of random matrices and let $Z_1^{(m)},\dots,Z_p^{(m)}$ and $U_1^{(m)},\dots,U_q^{(m)}$ be independent random matrices, such that 

(1). For each $i\leq p$, $X_i^{(m)}=P_{j(i)}^{(m)}X_i^{(m)}P_{j(i)}^{(m)}$ is a non-self-adjoint Wigner Matrix when restricted to the compressed space $P_{j(i)}^{(m)}\cA_{Nm} P_{j(i)}^{(m)}$ for some $1\leq j(i)\leq k$.

(2). For each $i\leq q$, $U_i^{(N)}=P_{h(i)}^{(N)}U_i^{(N)}P_{h(i)}^{(N)}$ is a Haar-unitary random matrix when restricted to the compressed space $P_{h(i)}\cA_{Nm} P_{h(i)} \otimes \cM_m(\CC)$, for some $1\leq h(i)\leq k$. Then
$$(X_1^{(m)},\dots,X_p^{(m)},U_1^{(m)},\dots,U_q^{(m)},D_1^{(m)},\dots,D_r^{(m)})\to(c_1,\dots,c_p,u_1,\dots u_q,D_1,\dots D_r)$$ where $c_1,\dots,c_p,u_1,\dots u_q,D_1,\dots D_r$ are elements in a rectangular probability space $(\cA,\tau)$ (with $(D_1,\dots,D_r,P_1,\dots,P_k)\subset\cM_N(\CC)\subset \cA$) and such that: 

(i). For all $i\leq p$, $c_i=P_{j(i)}c_iP_{j(i)}$ is circular in the compressed space $P_{j(i)} \cA P_{j(i)}$.

(ii). For all $i\leq q$, $u_i=p_{h(i)}u_ip_{h(i)}$ is a Haar-unitary in the compressed space $P_{h(i)}\cA (\CC) P_{h(i)}$.

(iii). The algebras $\langle c_1\rangle,\dots,\langle c_p\rangle,\langle u_1,u_1^*\rangle,\dots, \langle u_q,u_q^*\rangle,\langle D_1,D_1^*\dots, D_r,D_r^* \rangle $ are free with amalgamation over $\langle P_1,\dots P_k\rangle$.

In particular, for any fixed non-commutative polynomial $$Q(z_1,z_1^*\dots,z_p,z_p^*,y_1,y_1^*,\dots,y_q,y_q^*,w_1,\dots,y_r),$$ such that the ensemble $$Q_m:=Q(Z_1^{(m)},(Z_1^{(m)})^*\dots,(Z_p^{(m)})^*,U_1^{(m)},(U_1^{(m)})^*\dots,(U_q^{(m)})^*,D_1^{(m)},\dots,D_r^{(m)})$$ is self-adjoint (i.e. $Q_m=Q_m^*$), then $Q_m$ converges in AED to its Free Deterministic Equivalent $$Q_{\infty}:=Q(c_1,c_1^*\dots,c_p,c_p^*,u_1,u_1^*,\dots,u_q,u_q^*,D_1,\dots,D_r)$$    
\end{theoremdefinition}

In order to perform numerical computations, it will be convenient to supress unitaries whenever we deal with randomly rotated deterministic matrices.

\subsection{Rotated families of deterministic matrices and elimination of unitaries}\label{rotations}

If sub-algebras $\langle A, B\rangle\subset\cA$ are free from $\langle u, u^*\rangle$, where $u$ is a Haar unitary , and $(a_1,\dots ,a_p)\in A^p$, $(b_1,\dots ,b_q)\in B^q$, then the joint distribution of  $(\tilde a_1,\dots,\tilde a_p,b_1,\dots b_q)$, where $\tilde a_i=ua_iu^*$ is completely determined: 
$\{\tilde a_1,\dots,\tilde a_p\}$ and $\{b_1,\dots b_q\}$ are free and $(\tilde a_1,\dots,\tilde a_p)\sim (a_1,\dots,a_p)$. 
Roughly speaking, conjugating a family of variables by a free Haar-unitary does not alter the distribution of the family and makes it free from a second family of variables.

The fact that $(\tilde a_1,\dots,\tilde a_p)\sim (a_1,\dots,a_p)$ is trivial by the tracial property of $\tau$. 
Hence we only need to show that $\langle ua_1u^*,\dots ,ua_pu^*\rangle$ and $\langle b_1,\dots,b_q\rangle$ are free.

Let $a^{(1)},\dots,a^{(k)}\in \langle a_1,\dots,a_p\rangle$ and $b^{(1)},\dots,b^{(k)}\in \langle b_1,\dots,b_q\rangle$. 
We note that $a^{(j)}\in \langle a_1,\dots,a_p\rangle$ iff $ua^{(j)}u^*\in \langle ua_1u^*,\dots,ua_pu^*\rangle$. 
Since $\tau(ua^{(j)}u^*)=\tau(a^{(j)})$ and $\tau(u)=0=\tau(u^*)$, we have that $\overline{uau^*}=\overline{u}\overline{a^{(j)}}\overline{u^*}$. 

Hence 
\begin{eqnarray*}
\tau((\overline{ua^{(1)}u^*})\overline{b^{(1)}}\cdots (\overline{ua^{(k)}u^*})\overline {b^{(k)}})&=&\tau(\overline{u} \overline{{a^{(1)}}}\overline {u^*}\overline{b^{(1)}}\cdots \overline{u}\overline{a^{(k)}}\overline{u^*}\overline{b^{(k)}}),\\
\tau( \overline{b^{(1)}}(\overline{ua^{(1)}u^*})\cdots \overline{ b^{(k)}}(\overline{ua^{(k)}u^*}))&=&\tau(\overline{b^{(1)}}\overline{u}\overline{a^{(1)}}\overline{u^*}\cdots \overline{u}\overline{a^{(k-1)}}\overline{u^*}\overline{b^{(k)}}),\\
\tau((\overline{ua^{(1)}u^*}) \overline{b^{(1)}}\cdots \overline{b^{(k-1)}}(\overline{ua^{(k)}u^*}))&=&\tau(\overline{u}\overline{a^{(1)}}\overline{u^*}\overline{b^{(1)}}\cdots \overline{b^{(k-1)}}\overline{u}\overline{a^{(k)}}\overline{u^*}),\\
\tau(\overline{ b^{(1)}}(\overline{ua^{(1)}u^*})\cdots (\overline{ua^{(k-1)}u^*})\overline{ b^{(k)}})&=&\tau(\overline{b^{(1)}}\overline{u}\overline{a^{(1)}}\overline{u^*}\cdots \overline{u}\overline{a^{(k-1)}}\overline{u^*}\overline{b^{(k)}}).                                                                                                                              
\end{eqnarray*}
By freeness of $\langle a_1,\dots,a_p,b_1,  \dots,b_q\rangle$ and $\{u,u^*\}$ all the RHS expressions vanish and the freeness of $\langle ua_1u^*,\dots ,ua_pu^*\rangle$ and $\langle b_1,\dots,b_q\rangle$ is established.

The same statement (with the same proof) holds for several randomly rotated collections:

\begin{prop}\label{prop:rot}
Let $A_0,\dots,A_k\subseteq \cA$ be $*$-sub-algebras of a $*$-probability space $(\cA,\tau)$ and let $u_1,\dots,u_k\in\cA$ be Haar-unitary elements, such that $\langle A_0,\dots,A_k\rangle$, $\langle u_1,u_1^*\rangle,\dots,\langle u_k,u_k^*\rangle$ are free. 
For $0\leq j\leq k$, let $(a_1^{(j)},a_2^{(j)},\dots ,a_{p(j)}^{(j)})\in A_j^{p(j)}$. Then 
\begin{eqnarray}
&&(a_1^{(0)},\dots,a_{p(0)}^{(0)},u_1a_1^{(1)}u_1^*,\dots,u_1a_{p(1)}^{(1)}u_1^*,\dots,u_ka_{p(k)}^{(k)}u_k^*)\\ &\sim&(a_1^{(0)},\dots,a_{p(0)}^{(0)},\tilde a_1^{(1)},\dots,\tilde a_{p(1)}^{(1)},\dots,\tilde a_{p(k)}^{(k)}),                                                                                                                                                                                       
\end{eqnarray}
where $\langle a_1^{(0)},\dots,a_{p(0)}^{(0)}\rangle,\langle \tilde a_1^{(1)},\dots,\tilde a_{p(1)}^{(1)}\rangle,\dots,\langle \tilde a_{1}^{(k)},\dots \tilde a_{p(k)}^{(k)}\rangle$ are free and $(\tilde a_{1}^{(j)},\dots \tilde a_{p(j)}^{(j)})\sim(a_{1}^{(j)},\dots a_{p(j)}^{(j)})$, for $j\leq k$.
\end{prop}

By the previous proposition we may, for example, reduce the problem of investigating the distribution of $a+ubu^*$ to that of the sum of free copies $\tilde a + \tilde b$.

\begin{figure}
 \centering
\setlength{\unitlength}{.4mm}
\begin{picture}(120,120)
\put(0,0){\line(0,1){120}}
\put(25,0){\line(0,1){120}}
\put(120,0){\line(0,1){120}}
\put(60,95){\line(0,1){25}}
\put(60,60){\line(0,1){35}}
\put(75,0){\line(0,1){45}}
\put(75,95){\line(0,1){25}}
\put(5,105){$P_0$}
\put(29,83){$ T_1, P_1$}
\put(29,67){$ u_1, u_1^*$}
\put(35,105){$ R_1$}
\put(92,105){$ R_k$}
\put(5,75){$ R_1^*$}
\put(5,20){$  R_k^*$}
\put(62,48){$\ddots$}
\put(10,48){$\vdots$}
\put(62,105){$\dots$}
\put(83,28){$ T_k, P_k$}
\put(83,12){$ u_k,  u_k^*$}
\put(0,0){\line(1,0){120}}
\put(0,95){\line(1,0){120}}
\put(0,120){\line(1,0){120}}
\put(25,60){\line(1,0){35}}
\put(0,60){\line(1,0){25}}
\put(75,45){\line(1,0){45}}
\put(0,45){\line(1,0){25}}
\end{picture}
\caption{\text{Embedding of the operators forming $Q=Q_{\infty}$ on a Rectangular Space}}
    \label{fig:rectspace}
\end{figure}

For our example $Q_{\infty}=\sum_{i=1}^k R_iu_iT_iu_i^*R_i^*,$ it is convenient to think that the operators forming $Q_{\infty}$ (as well as the matrices forming $Q_m$), are embedded in a rectangular probability space $\cA$, as illustrated in Fig. \ref{fig:rectspace} (where $\tilde N:=N+n_1+\dots+n_k$ and $P_0,\dots,P_k\in M_{\tilde N}\subset \cA$ are such that $Tr(P_0)=N$ and $Tr(P_i)=n_i$). Let $R=\sum_{i=1}^k R_i$, $U=\sum_{i=1}^k u_i$ and $T=\sum_{i=1}^k T_i$, then we have that the desired distribution (of $Q_{\infty}$) is just that of the element $RUTU^*R^*$ in the compressed space $P_0 \cA P_0$. Just as in the scalar case, where the unitaries conjugating a variable may be removed, provided that the conjugated variable becomes free from the rest, the element  $RUTU^*R^*$ can be thought as $\tilde R \tilde T \tilde R$, where $\tilde R, \tilde T$ are free over $\langle P_0,\dots,P_k\rangle$.

\begin{prop}\label{recrot}
Let $(\cA,\FF)$ be a $\cD:=\langle p_1,\dots,p_k\rangle$rectangular probability space. Let $A_1,A_2\subset \cA$ and $U=u_1+u_2+\dots+u_k\in \cA$ be such that $\langle D_1,D_2\rangle,U$ are $\cD$-free and $u_j=p_ju_jp_j$ is a Haar unitary in the compressed space $p_j\cA p_j$. Then $D_1,UD_2U^*$ are $\cD$-free.
\end{prop}
\begin{proof}
Just replace $\tau$ by $\FF$ in Prop. \ref{prop:rot}.
\end{proof}

In \cite{SpVa12} we showed that the Cauchy Transforms of our FDE's satisfy the Equations in \cite{CoDe11}. In order to re-derive their equations, we needed to understand how the mixed moments of free random variables are calculated. In the next section we list the main tools for showing this correspondence.

\subsection{Combinatorics and Cumulants}\label{combi}

Recall Definition \ref{partitions}. A partition $\pi\in\SP(n)$ is \emph{non-crossing} if there is no
quadruple of elements $1\leq i<j<k<l\leq n$ such that $i\sim_\pi k$, $j\sim_\pi l$ and
$i\not\sim_\pi j$. The non-crossing partitions of order $n$ form a
sub-poset of $\SP(n)$ which we denote by $\NCP(n)$.

For $n\in\mathbb{N}$, a $\mathbb{C}$-multi-linear map $f:\mathcal{A}^{n}\to\mathcal{B}$ is called $\mathcal{B}$\textit{-balanced} if it satisfies the $\mathcal{B}$-bilinearity conditions, that for all $b,b'\in\mathcal{B}$, $a_{1},\dots,a_{n}\in\mathcal{A}$, and for all $r=1,\dots,n-1$
\begin{eqnarray*}
f\left(ba_{1},\dots,a_{n}b'\right) &=& bf\left(a_{1},\dots,a_{n}\right)b'\\
f\left(a_{1},\dots,a_{r}b,a_{r+1},\dots,a_{n}\right) &=& f\left(a_{1},\dots,a_{r},ba_{r+1}\dots,a_{n}\right)
\end{eqnarray*}

A collection of $\mathcal{B}$-balanced maps $\left(f_{\pi}\right)_{\pi\in \NCP}$ is said to be {multiplicative} with respect to the lattice of non-crossing partitions if, for every $\pi\in \NCP$, $f_{\pi}$ is computed using the block structure of $\pi$ in the following way:

\begin{itemize}
 \item If $\pi=\hat{1}_{n}\in \NCP\left(n\right)$, we just write $f_{n}:=f_{\pi}$.

 \item If $\hat{1}_{n}\neq\pi=\left\{ V_{1},\dots,V_{k}\right\} \in \NCP\left(n\right),$ then by a known characterization of $\NCP$, there exists a block $V_{r}=\left\{ s+1,\dots,s+l\right\} $ containing consecutive elements. For any such a block we must have
\begin{equation*}
f_{\pi}\left(a_{1},\dots,a_{n}\right)=f_{\pi\backslash V_{r}}\left(a_{1},\dots,a_{s}f_{l}\left(a_{s+1},\dots,a_{s+l}\right),a_{s+l+1},\dots,a_{n}\right),
\end{equation*}
where $\pi\backslash V_{r}\in \NCP\left(n-l\right)$ is the partition obtained from removing the block $V_{r}$.

\end{itemize}

The {operator-valued free cumulants} $\left(R^{\cB}_{\pi}\right)_{\pi\in \NCP}$ are defined as the unique multiplicative family of $\mathcal{B}$-balanced maps satisfying the (operator-valued) moment-cumulant formulas
\begin{equation*}
\EE\left(a_{1}\dots a_{n}\right)=\sum_{\pi\in \NCP\left(n\right)}R^{\cB}_{\pi}\left(a_{1},\dots,a_{n}\right)
\end{equation*}

By the {cumulants of a tuple} $(a_{1},\dots,a_{k})\in\mathcal{A}^k$, we mean the collection of all cumulant maps
\begin{equation*}
\begin{array}{cccc}
R_{i_{1},\dots,i_{n}}^{\mathcal{B};a_{1},\dots,a_{k}}: & \mathcal{B}^{n-1} & \to & \mathcal{B},\\
 & \left(b_{1},\dots,b_{n-1}\right) & \mapsto & R^{\cB}_{n}\left(a_{i_{1}},b_{1}a_{i_2},\dots,b_{i_{n-1}}a_{i_n}\right)\end{array}
\end{equation*}
for $n\in\mathbb{N}$, $1\leq i_{1},\dots,i_{n}\leq k$.

A cumulant map $R_{i_{1},\dots,i_{n}}^{\mathcal{B};a_{1},\dots,a_{k}}$ is {mixed} if there exists $r<n$ such that $i_{r}\ne i_{r+1}$. The main feature of the operator-valued cumulants is that they characterize freeness with amalgamation:

\begin{theorem}[\cite{Sp98}]
The random variables $a_{1},\dots,a_{n}$ are $\mathcal{B}$-free iff all their mixed cumulants vanish.
\end{theorem}

Other important combinatorial tools (which in particular are used to re-derive the formulas for the DE's from FDE's) are the formulas for computing cumulants of products (see \cite{KrSp00}, \cite{Sp00}) and the characterizations of freeness at different levels (\cite{NSS02}) in terms of cumulants.

\subsection{The analytic subordination phenomena}
Like in the scalar case, there are analytical tools to compute operator-valued free convolutions, which are based on the $\mathcal{B}$-valued Cauchy-transform $$G_x^{\mathcal{B}}(b)=\FF((b-x)^{-1}),$$ which maps the operatorial upper half-plane $$\mathbb{H}^+(\mathcal{B}):=\{b\in\mathcal{B}:\exists \ee> 0 \text{ such that }-i(b-b^*)\geq \ee\cdot 1\}$$ into the lower half-plane $\mathbb{H}^-(\mathcal{B})=-\mathbb{H}^+(\mathcal{B})$. 
In the usual settings coming from random matrix models (as we have seen above), our probability space $\cA$ may have several operator-valued structures $\FF_i:\cA\to\cB_i$ simultaneously, with $\CC=\cB_1\subset\cB_2\subset\dots\subset\cB_k$, and $\FF_i\circ \FF_{i+1}=\FF_i$. 
We are usually interested ultimately in the scalar-valued distribution, which can be obtained (via Stieltjes inversion) from the Cauchy-transform. 
The later in turn can be obtained from any ''upper'' $\mathcal{B}_i$-valued Cauchy transform, as we have that, for all $b\in\cB_i$ $$\FF_i(G_x^{\mathcal{B}_{i+1}}(b))=\FF_i\circ \FF_{i+1}((b-x)^{-1})=\FF_i((b-x)^{-1})=G^{\cB_i}_x(b).$$

A drawback of the operator-valued setting is that, unless we ask $\mathcal{B}$ to be commutative, one can hardly compute explicit distributions: 
although $\mathcal{B}$-valued generalizations of the $R$ and $S$-transforms exist (\cite{Vo95}, \cite{Dy06}), the task of explicitly inverting these operator-valued analytic maps is nearly impossible for any non-trivial situation (even for finite dimensional, relatively simple sub-algebras, like $\cB=M_2(\mathbb{C})$).

In terms of moments, the operator-valued Cauchy transform is given by $$G_x^{\mathcal{B}}(b)=\FF((b-x)^{-1})=\sum_{n\geq 0}\FF(b^{-1}(xb^{-1})^n)$$

The {operator-valued $\mathcal{R}$-transform} is defined by
\begin{equation*}
\mathcal{R}_{x}^{\mathcal{B}}\left(b\right)=\sum_{n\geq1}R_{n}^{\mathcal{B}}\left(x,bx,\dots,bx\right). 
\end{equation*}
The vanishing of mixed cumulants for free variables implies the additivity of the cumulants, and thus also the additivity of the $\cR$-transforms \cite{Vo95}: If $a_1$ and $a_2$ are $\cB$-free then we have for $b\in \cB$ that 
$\cR_{a_1+a_2}(b)=\cR_{a_1}(b)+\cR_{a_2}(b)$.

These transforms satisfy the functional equation
\begin{equation}
G_{a}^{\mathcal{B}}\left(b\right)=\left(\mathcal{R}_{a}^{\mathcal{B}}\left(G_{a}^{\mathcal{B}}\left(b\right)\right)-b\right)^{-1} \label{GR},
\end{equation}
which was crucial in \cite{SpVa12} to derive the Equations \ref{CoDeGau} and hence show the correspondence between DE's and FDE's.

Rather than using directly the $\mathcal{R}$-transform, a very powerful method to obtain $\mathcal{B}$-valued free convolutions is based on the analytic subordination phenomena observed by Biane (\cite{Bi98}, see also \cite{Vo02}). In particular, the approach of \cite{BeBe07} to obtain the subordination functions by iterating analytic maps can be very efficiently performed in the $\mathcal{B}$-valued context.

\begin{theorem}\cite{BMS13} \label{baddconv}
Let $(\cA,\FF)$ be a $\cB$-valued $C^*$-probability space and let $x,y\in\mathcal{A}$ be self-adjoint, $\mathcal{B}$-free. 
There exist an analytic map $\omega:\mathbb{H}^+(\mathcal{B})\to \mathbb{H}^+(\mathcal{B})$ such that $G_x(\omega(b))=G_{x+y}(b)$. 
Furthermore, for any $b\in\mathbb{H}^+(\mathcal{B})$ the subordination function $\omega(b)$ satisfies $$\omega(b)=\lim_{n\to\infty}f^{\circ n}_b(w),$$ where, for any $b,w\in \mathbb{H}^+(\mathcal{B})$, $f_b(w)=h_y(h_x(w)+b)+b$ and $h$ is the auxiliary analytic self-map $h_x(b)=(E((b-x)^{-1}))^{-1}-b$ on $\mathbb{H}^+(\mathcal{B})$. 
\end{theorem}

Numerically speaking, going from $h_x$ to $G_x$ and vice-versa is a simple operation. 
This means that one only needs the individual $\mathcal{B}$-valued Cauchy transforms of $x,y$ (or good approximations of these) to obtain the $\mathcal{B}$-valued Cauchy transform of $x+y$, and hence, its probability distribution. 
The operator-valued multiplicative convolution can also be numerically approximated (see \cite{BSTV14}).

In order to implement our main algorithm, we will be interested in the situation described in Prop. \ref{matrixfreeness} where $(\cA,\FF)$ is a rectangular probability space and hence our main space $(M_n(\CC)\otimes\cA,id\otimes \FF)$ consists of $n\times n$ matrices with entries in $\cA$, endowed with the entry-wise evaluation of $\FF:\cA\to \cB$. In Section \ref{SecLinTri} we will use Anderson's self-adjoint linearization trick to obtain the distribution of a polynomial on $\cB$-free variables (such as $\tilde R$ and $\tilde T$, which form $Q_{\infty}$) from the $(M_n(\CC)\otimes \cB)$-distribution of a specially constructed operator, which depends linearly on the inputs of the polynomial. 

In the next section we show how to obtain the $M_n(\cB)$-valued Cauchy-transforms of such linear elements.

\subsection{Linear elements}\label{Linear}

In a scalar-valued non-commutative probability space $(\cA,\tau)$, we have the integral representation of the Cauchy-transform: 
$$G_x(z)=\tau((z-x)^{-1})=\int_{\RR}(z-t)^{-1}\mathrm{d}\mu_x(t).$$
Analogously, for linear, self-adjoint elements $D\otimes x$ in a $M_n(\CC)$-valued probability space $(M_n(\CC)\otimes(\cA),id_m\otimes \tau)$, we have: 
$$G_{D\otimes x}(b)=(id_m\otimes \tau)((b-D\otimes x)^{-1})=\int_{\RR}(b-D\otimes t)^{-1}\mathrm{d}\mu_x(t).$$
The previous integrals can be approximated, for example, by using matrix-valued Riemann sums. In particular, we are able to approximate the $M_n(\CC)$-valued Cauchy transform of any self-adjoint matrix which depends linearly on a semicircular element $s$ (and hence also for a circular element $c$ which can be viewed as $s_1+is_2$, for free semi-circulars $s_1,s_2$). The same can be done if we start with a rectangular probability space.

Let $(\cA, \FF)$ be a $\langle p_1,\dots p_k\rangle$-rectangular probability space and consider the $\cB$-valued probability space $( M_m(\CC) \otimes\cA, \FF_2)$, where $\FF_2= id_m \otimes \FF$ and $\cB=(M_m(\CC)\otimes \langle p_1,\dots p_k\rangle)$ 

Consider $x\in \cA$ of the form $x=\alpha_1 p_1s_1p_1+\dots+\alpha_k p_ks_kp_k$, where $s_i=p_is_ip_i$ is a semicircular element when restricted to $\cA^{(i)}$. Let $D\in M_m(\CC)$, and let $b=(b_{ij})_{i,j\leq m}\in \cB$, with $b_{ij}=\beta^{ij}_1p_1+\dots+\beta^{ij}_kp_k$. Then we have
\begin{eqnarray*}
G^{\cB}_{D\otimes x}(b) &=&\FF_2((b-D\otimes x)^{-1})\\
&=&(id_m\otimes \FF )((b-D\otimes x)^{-1})\\
&=&\frac{1}{2\pi}\int_{-2}^2[((\beta^{ij}_1-D_{ij}\alpha_1t)p_1+\dots+(\beta^{ij}_k-D_{ij}\alpha_kt)p_k)_{ij}]^{-1}\sqrt{4-t^2}dt.
\end{eqnarray*}

The case of deterministic matrices is simpler. If we assume that $M_n(\CC)\subset\cA$ and consider $D=D^*\in M_m(\CC)\otimes M_n(\CC)$. Then $G^{\cB}_{D}(b)=G^{\cB}_{D}(b\otimes I_n)$ is just the partial trace $(id_m\otimes \FF_2)((b\otimes I_n-D)^{-1})$.

One should be able to provide a similar trick to approximate Cauchy transforms for elements of the form $D\otimes u+ D^*\otimes u^*$. For the moment, we find a way around this problem by removing Haar unitaries, as discussed in Section \ref{rotations}.

\section{The linearization trick and the main algorithm}\label{SecLinTri}

One of the main ingredients of our algorithm was already suggested by Voiculescu in his earlier papers on operator-valued free probability: 
the possibility to transfer questions about the distribution of a polynomial in non-commutative random variables to a question about the matrix-valued distribution of a related polynomial with matrix-valued coefficients, with the advantage of being linear on the non-commutative variables.

The idea was formalized and put into practice by Haagerup and Thorbjornsen \cite{HaTh05}. Some years later, Anderson \cite{An11} found linearizations which preserve self-adjointness properties, based on Schur complements. 
In the next section we generalize Anderson's self-adjoint linearization trick to be able to deal with operator-valued situations. 

Our machinery to deal with matricial and rectangular distributions is very well behaved with respect to the different elements of the numerical algorithm, developed in \cite{BMS13}, to compute distributions of self-adjoint polynomials on free self-adjoint random variables. For this reason, we will only point out those few steps where our situation differs.

Later, we describe the FDE's for the models in \cite{CoDe11} and suggest some numerically efficient linearizations.

\begin{prop}
Let $(\cA,\cB)$ be a $\cB$-probability space and let $x_1,\dots,x_{n}\in \cA$.
Let $P=P(x_1,\dots,x_{n})\in \cB\langle x_1,\dots x_n,x_1^*,\dots x_n^* \rangle$ be a self-adjoint $\cB$-valued polynomial in $x_1,\dots,x_{n}$ and their adjoints.
There exist $m\geq 1$ and an element $L_P\in M_m(\CC)\otimes\cA$ such that:
\begin{enumerate}
\item $L_P=c_1\otimes x_1+c_1^*\otimes x_1^*+\dots c_{n}\otimes x_{n}+c_{n}^*\otimes x_{n}^*+c \in M_m(\CC)\otimes\cA$, with $c\in M_m(\CC)\otimes \cB$ and, for $i\geq 1$ $c_i\in M_m(\CC)$.
\item If $\Lambda_{\varepsilon}(b))=diag(b,i\varepsilon,i\varepsilon,\dots,i\varepsilon)\in M_m(\CC)\otimes \cB$, then $$G^{\cB}_{P}(b)=\lim_{\varepsilon\downarrow0}(G^{M_m(\CC)\otimes \cB)}_{L_P}(\Lambda_{\varepsilon}(b)))_{11}.$$

\end{enumerate} 
\end{prop}

\begin{proof}
The main idea is to think of the polynomial $P\in \cB\langle x_1,\dots,x_n,x_1^*,\dots, x_n^*\rangle$ as a polynomial $P\in \CC\langle x_1,\dots,x_n,x_{n+1},\dots, x_{2n},b_1,\dots,b_s\rangle$, where $x_{n+j}=x_j^*$ and the $b_i$'s are the elements of $\cB$ which actually appear as coefficients in $P$. With this, we are able to use [\cite{BMS13}, Prop. 3.2, Cor. 3.3 and Prop. 3.4].

Note that, by proceeding as in [\cite{BMS13}, Cor. 3.5], we will also get a self-adjoint linearization $$L_P=c_1\otimes x_1+\dots +c_{n}\otimes x_{n}+d_1\otimes x_1^*+\dots+d_{n}\otimes x_{n}^*+e_1\otimes b_1+\dots+e_s\otimes b_s+f.$$

The fact that $L_P=L_P^*$ will mean of course that $d_i=c_i^*$ and $c^*=c:=e_1\otimes b_1+\dots+e_s\otimes b_s+f$. So our linearization has the desired form.

In view of [\cite{BMS13}, Cor. 3.3], one has again that $$(b-P)^{-1}=[(\Lambda_0(b)-L_P)^{-1}]_{11}$$ whenever $(b-P)$ (or, equivalently $\Lambda_0(b)-L_P$) is invertible. Hence, the linearization works actually at the level of resolvents and the translation to Cauchy-transforms is obtained by applying $id_m\otimes \FF$ to the resolvent of the right side (we must, however, consider $\Lambda_{\varepsilon}(b)$ as in [\cite{BMS13}, Cor. 3.6] so that the argument belongs to the operatorial upper-half-plane, which is the right domain of the Cauchy-transform for a later application of Theorem. \ref{baddconv}).  
\end{proof}

We include below the adaptations of [\cite{BMS13}, Prop 3.4 and Cor. 3.5] to our situation, which provide such linearizations.

\begin{remark} We recall one procedure to obtain a self-adjoint linearization.
A general monomial $P=b_0x_{i_1}b_1\cdots x_{i_k}b_k$ has a (possibly non-self-adjoint) linearization
$$L_P=\left[\begin{array}{cccccc}
& & & & & b_0\\
& & & & x_{i_1} & -1 \\
& & & b_1 & -1 & \\
& & \dots & \dots & & \\ 
& x_{i_k} & -1 & & & \\
b_k & -1 & & & &                         
\end{array}
 \right]$$

If $P=P_1+\dots+P_k$ and each $P_j$ has a linearization 
$$L_{P_j}=\left[\begin{array}{cc}
0 & u_j\\
v_j & Q_j                    
\end{array}
\right],$$
then a linearization of $P$ is given by 
$$L_P=
\left[\begin{array}{cccc}
0 & u_1 & \cdots  & u_k\\
v_1 & & &   \\
\vdots & & \ddots & \\
v_k & & & Q_k 
\end{array}
 \right].$$

Finally, if $P$ is self-adjoint, we may view it as $P=q+q^*$ for $q=P/2$. If 
$$L_{q}=\left[\begin{array}{cc}
0 & u\\
v & Q                    
\end{array}
\right]$$ is a linearization of $q$ then
$$L_{P}=\left[\begin{array}{ccc}
0 & u & v^*\\
u^* & 0 & Q^*\\
v & Q & 0                    
\end{array}
\right]$$ is a self-adjoint linearization of $P$.
\end{remark}

\begin{remark}
Since we are able to compute operator-valued Cauchy transforms of arbitrary deterministic matrices (as these are just partial traces), the products of deterministic matrices do not really bother us. We should use the linearization trick only to transform the polynomial into a polynomial with matrix coefficients which is linear in the variables which correspond to random matrices but needs not necessarily to be linear on the variables corresponding to deterministic matrices.
\end{remark}

\subsection{Examples from wireless communications}

Now we consider some matrix models from \cite{CoDe11}. Understanding these models and their deterministic equivalents was one of the main motivations of our work. After each model, we discuss embeddings of the matrices in rectangular spaces, we then discuss the FDE and we give a linearization which allows to plot the distribution.

\subsubsection{Unitary precoded channels \cite{CHD11}}\label{Exwire1}

For the model $$Q=\sum_{i=1}^k R_iU_iT_iU_i^*R_i^*,$$
we already discussed its embedding in a rectangular space and its FDE $Q_N=\tilde R \tilde T \tilde R ^*$, so we are only missing its linearization. It is very simple, namely
$$L_{P_{\infty}}=\left[\begin{array}{ccc}
0 & 0 & \tilde R\\
0 & \tilde T & -1 \\
 \tilde R^* & -1 & 0 
\end{array}
 \right],$$
where each entry is really an $\tilde N\times \tilde N$ block, with $\tilde N=N_0+n_1\dots+n_k$.
The individual $M_3(\cB)$-valued Cauchy transforms of the self-adjoint elements 
$$L_1=\left[\begin{array}{ccc}
0 & 0 & \tilde R\\
0 & 0 & -1 \\
\tilde R^* & -1 & 0 
\end{array}
 \right],\quad L_2=\left[\begin{array}{ccc}
0 & 0 &  0\\
0 & \tilde T & 0 \\
0 & 0 & 0 
\end{array}
 \right],$$ 
can be computed by performing partial traces, as explained in Section \ref{Linear}.

Fig. 2 in the Introduction shows the implementation of our algorithm for this case.

\subsubsection{Correlated MIMO multiple access channels \cite{CDS11}}\label{Exwire2}. 

Let us go back to the model
$$P=\sum_{i=1}^k R_iZ_iT_iZ_i^*R_i.$$
In order to achieve asymptotic freeness we embed the matrices in a rectangular space exactly as we did in the previous case. The FDE will be then  $$P_{\infty}=\sum_{i=1}^k R_i c_iT_ic_i^*R_i^*,$$ 
and the linearization yields
$$L_{P_{\infty}}=\left[\begin{array}{ccccc}
0 & 0 & 0 & 0 & R\\
0 & 0 & 0 & c & -1\\
0 & 0 & T & -1 & 0 \\
0 & c^* & -1 & 0 & 0 \\
R^* & -1 & 0 & 0 & 0 
\end{array}
 \right],$$
where each entry is again $\tilde N\times \tilde N$, with $\tilde N=N+n_1\dots+n_k$.
The individual $M_5(\cB)$-valued Cauchy-transforms of the self-adjoint elements 
$$L_1=\left[\begin{array}{ccccc}
0 & 0 & 0 & 0 & R\\
0 & 0 & 0 & 0 & -1\\
0 & 0 & T & -1 & 0 \\
0 & 0 & -1 & 0 & 0 \\
R^* & -1 & 0 & 0 & 0 
\end{array}
 \right],\quad L_2=\left[\begin{array}{ccccc}
0 & 0 & 0 & 0 & 0\\
0 & 0 & 0 & c & 0\\
0 & 0 & 0 & 0 & 0 \\
0 & c^* & 0 & 0 & 0 \\
0 & 0 & 0 & 0 & 0 
\end{array}
 \right],$$ 
can be computed by, respectively, performing a partial trace and approximating by matrix-valued Riemann sums (or, alternatively, by using the method in \cite{HRS07}), as explained in Section \ref{Linear}.

\subsubsection{Frequency selective MIMO systems \cite{DuLo07}} \label{Exwire3}
Let $k\geq1$ be fixed again and consider now the model
$$P_N=\sum_{i=1}^k (R_iZ_iT_i)\sum_{j=1}^k(T_j^*Z_j^*R_j^*),$$
where the individual matrices are as in the previous example and additionally $n_i=n$ for all $i\leq k$.

We embed again $R_i,Z_i$ in a rectangular space in such a way that $P_i R_i P_0= R_i$ and $P_i Z_i P_i= Z_i$, but this time we put $P_0  T_i P_i= T_i$.

If (as in the first example), we put again $R=\sum R_i$, $T=\sum T_i$, $c=\sum c_i$, our FDE can be compactly written as $$P_{\infty}=RcTT^*c^*R^*.$$
The linearization will be very similar to the one in the previous case, the main difference is the way in which we have embedded the matrices $T_i$. We get

$$L_{P_{\infty}}=\left[\begin{array}{ccccc}
0 & 0 & 0 & 0 & R\\
0 & 0 & 0 & c & -1\\
0 & 0 & TT^* & -1 & 0 \\
0 & c^* & -1 & 0 & 0 \\
R^* & -1 & 0 & 0 & 0 
\end{array}
 \right],$$
where each entry is $\tilde N \times \tilde N$, with $\tilde N=N+kn$.
The individual $M_5(\cD)$-valued Cauchy-transforms of the self-adjoint elements 
$$L_1=\left[\begin{array}{ccccc}
0 & 0 & 0 & 0 & R\\
0 & 0 & 0 & 0 & -1\\
0 & 0 & TT^* & -1 & 0 \\
0 & 0 & -1 & 0 & 0 \\
R^* & -1 & 0 & 0 & 0 
\end{array}
 \right],\quad L_2=\left[\begin{array}{ccccc}
0 & 0 & 0 & 0 & 0\\
0 & 0 & 0 & c & 0\\
0 & 0 & 0 & 0 & 0 \\
0 & c^* & 0 & 0 & 0 \\
0 & 0 & 0 & 0 & 0 
\end{array}
 \right],$$ 
can be again computed as explained in Section \ref{Linear}.
\bibliography{Doktorarbeit}
\bibliographystyle{amsplain}

\end{document}